\documentclass[11pt]{article}
\pdfoutput=1
\usepackage{listings}
\usepackage{pdflscape}
\usepackage{amsfonts}
\usepackage{epsfig}
\usepackage{lscape}
\usepackage{longtable}
\usepackage{amsmath,amssymb,amsthm}
\usepackage{graphicx}
\usepackage{subcaption}
\usepackage{color}
\usepackage{placeins}
\usepackage{url}
\usepackage{cases}
\usepackage{longtable}
\usepackage{supertabular}
\usepackage{multirow}

\usepackage{amsmath}
\allowdisplaybreaks[4]
\usepackage{cite}
\usepackage{algorithmic}
\usepackage{algorithm,float}
\usepackage{booktabs}
\usepackage{enumitem}
\usepackage{array}
\usepackage{lipsum}
\usepackage{hyperref}
\usepackage{cleveref}
\usepackage{svg}
\usepackage{xcolor}
\usepackage{subcaption}
\usepackage{authblk}

\usepackage[table]{xcolor}
\usepackage{booktabs}



\newtheorem{theorem}{Theorem}
\newtheorem{assumption}{Assumption}
\newtheorem{lemma}{Lemma}
\newtheorem{remark}{Remark}
\newtheorem{definition}{Definition}
\newtheorem{proposition}{Proposition}
\newtheorem{corollary}{Corollary}
\newtheorem{example}{Example}

\usepackage{caption}
\begin{document}

\title{\bf Efficient and provably convergent end-to-end training of deep neural networks with linear constraints}
\date{}

\author[1]{Zonglin Yang \textsuperscript{*}}
\author[1]{Zhexuan Gu \textsuperscript{*}}
\author[1]{Yancheng Yuan \textsuperscript{\dag}}
\affil[1]{Department of Applied Mathematics, The Hong Kong Polytechnic University, Hung Hom, Hong Kong}

\renewcommand{\thefootnote}{\fnsymbol{footnote}}

\footnotetext[1]{The authors contributed equally to this work.}
\footnotetext[2]{Corresponding author: {\tt yancheng.yuan@polyu.edu.hk}.}
\renewcommand{\thefootnote}{\arabic{footnote}}

\maketitle

\begin{abstract}
Training a deep neural network with the outputs of selected layers satisfying linear constraints is required in many contemporary data-driven applications. While this can be achieved by incorporating projection layers into the neural network, its end-to-end training remains challenging due to the lack of rigorous theory and efficient algorithms for backpropagation. A key difficulty in developing the theory and efficient algorithms for backpropagation arose from the nonsmoothness of the solution mapping of the projection layer. To address this bottleneck, we introduce an efficiently computable HS-Jacobian to the projection layer. Importantly, we prove that the HS-Jacobian is a conservative mapping for the projection operator onto the polyhedral set, enabling its seamless integration into the nonsmooth automatic differentiation framework for backpropagation. Therefore, many efficient algorithms, such as Adam, can be applied for end-to-end training of deep neural networks with linear constraints. Particularly, we establish convergence guarantees of the HS-Jacobian based Adam algorithm for training linearly constrained deep neural networks. Extensive experiment results on several important applications, including finance, computer vision, and network architecture design, demonstrate the superior performance of our method compared to other existing popular methods.
\end{abstract}

\newcommand{\keywords}[1]{%
  \par\vspace{0.5\baselineskip}%
  \noindent{\textbf{Keywords:} #1}
}


\section{Introduction}

Designing deep neural networks (DNNs) whose outputs of selected layers satisfy certain linear constraints has become increasingly important in contemporary applications, such as deep learning based portfolio allocation \cite{wang2023linsatnet}, deep graph matching \cite{wang2021neural}, and the topological architecture design of DNNs \cite{xie2025mhc}. Without loss of generality, assume that the output $x$ of a given layer in the DNN satisfies the linear constraints
\begin{equation}
\label{eq: polyhedral_set}
    \mathcal{P}:=\{y \in \mathbb{R}^n ~|~ Ay \le a, By=b\},
\end{equation}
where $A \in \mathbb{R}^{m \times n},a \in \mathbb{R}^m, B \in \mathbb{R}^{l \times n}, b \in \mathbb{R}^l$ are given and the matrix $B$ is of full row rank. This satisfaction can be guaranteed by incorporating a projection layer into the deep neural network with $\tilde{x}=\Pi_{\mathcal{P}}(x)$, where $\Pi_{\mathcal{P}} : \mathbb{R}^n \to \mathbb{R}^n$ is a projection operator onto $\mathcal{P}$
\begin{equation}
\label{eq: projector-general}
    \Pi_{\mathcal{P}}(x) := \underset{y\in\mathcal{P}}{\arg\min} ~ \frac{1}{2} \|y-x\|_2^2.
\end{equation}

In the end-to-end training of such DNNs, the forward pass of the projection layers can be efficiently computed by highly optimized quadratic programming (QP) solvers \cite{gurobi, stellato2018osqp, liang2022qppal, chen2025hpr}. However, how to efficiently compute the backpropagation with mathematical rigor remains open due to the lack of a valid chain rule. 

Several existing works have tried to tackle this challenge. One popular idea is to approximate the projection layer by some operators with valid automatic differentiation, but the feasibility of the obtained approximated solutions is usually poor. In particular, the unrolling-based approximation and the smooth penalty-based methods are commonly applied. On the one hand, the unrolling-based methods approximate a projection operator through multiple iterative update steps, which can be efficiently integrated into the DNNs with valid automatic differentiation. An important example is to include the Sinkhorn algorithm \cite{sinkhorn1967concerning, cuturi2013sinkhorn} for approximating the projection onto the Birkhoff polytope \cite{xie2025mhc}. Nevertheless, the unrolling-based methods often require substantially deeper networks to achieve higher approximation precision, thereby incurring expensive memory overhead and computational complexity \cite{amos2017optnet}. On the other hand, the penalty-based method approximates the projection operator via the solution mapping of an unconstrained optimization problem \cite{linghu2026penalty}, where the backpropagation is simulated through a smooth penalty function. However, this formulation inevitably introduces computational errors in the gradient computation with non-negligible computation overheads.

Another main approach is based on the (generalized) differentials of the Karush–Kuhn–Tucker (KKT) system of the projection problem \eqref{eq: projector-general} with the (generalized) implicit function theorem \cite{clarke1990optimization, fiacco1990nonlinear, rockafellar1998variational, sun2001further, dontchev2009implicit, mordukhovich2012variational, bolte2021nonsmooth}. One example is OptNet \cite{amos2017optnet}, where the forward pass involves solving a QP problem. This approach was later generalized to convex optimization layers \cite{agrawal2019differentiable}, where the backpropagation is performed by implicitly differentiating the optimality system of its corresponding conic program \cite{agrawal2019differentiating}. A similar method based on the optimality condition and the basic sensitivity theorem \cite{fiacco1976sensitivity} was also proposed for QP problems \cite{magoon2025differentiation}. However, the effectiveness of these methods depends heavily on strong assumptions regarding the input data, such as the strictly complementary condition \cite{barratt2018differentiability, magoon2025differentiation} or differentiability of the corresponding cone projection operator \cite{amos2019differentiable, bolte2021nonsmooth}, which can easily fail in practice. Moreover, when the generalized implicit theorem is applied, we need to find an element of the Clarke Jacobian \cite{clarke1975generalized, clarke1990optimization} of the projection onto a general polyhedral set, which is computationally expensive \cite{han1997newton, li2020efficient}. 

In this paper, we address the challenges for training linearly constrained DNNs by introducing the efficiently computable HS-Jacobian to the projection layer \cite{han1997newton,li2020efficient}, which can be seamlessly integrated into the nonsmooth automatic differentiation framework for backpropagation. Therefore, many efficient algorithms, such as stochastic gradient descent and Adam \cite{kingma2015adam}, can be applied to train linearly constrained DNNs. As an important ingredient, we prove that the HS-Jacobian constitutes a conservative mapping. As a side theoretical result of independent interest, we show the equivalence between the HS-Jacobian and the B-subdifferential of the projection over a polyhedral set under the linear independence constraint qualification (LICQ). We further establish the convergence of Adam with our HS-Jacobian based backpropagation for training DNNs with projection layers. The superior performance of our method has been demonstrated by extensive numerical experiment results on diversified applications across finance, computer vision, and topological neural network architecture design.

The main contributions of this paper are as follows:
\begin{enumerate}
    \item We introduce the HS-Jacobian to the projection layer and prove that it is a conservative mapping. This provides a fundamental building block for designing the backpropagation method for training linearly constrained DNNs. We also establish the equivalence between the HS-Jacobian and the B-subdifferential of the projection onto the polyhedral set under the LICQ, which is of independent interest.
    \item We propose the HS-Jacobian based backpropagation algorithm, which can be seamlessly integrated into efficient algorithms, such as stochastic gradient descent and Adam, to enable efficient training of DNNs with linear constraints.
    \item We establish the convergence of the proposed HS-Jacobian based Adam algorithm for training linearly constrained DNNs.
    \item Extensive numerical experiment results on diversified applications demonstrate the superior performance of the proposed method.
\end{enumerate}

In the remainder of the paper, we first introduce HS-Jacobian and the backpropagation method based on it. We then conduct the convergence analysis for training linearly constrained DNNs by the HS-Jacobian based Adam algorithm. After that, we will demonstrate the superior performance of this method with extensive numerical results. Lastly, we conclude the paper with some discussions for potential future research.

\section{The HS-Jacobian of the projection layer over a polytope}

In this section, we first present the definition of the HS-Jacobian \cite{han1997newton} of the projection layer. Then, we introduce an efficient method for computing an element of the HS-Jacobian, which can be used for backpropagation. The relationship of the HS-Jacobian and the B-subdifferential of the projection over a polyhedral set will be established, which is of independent interest. As a key theoretical result of this section, we will prove that the HS-Jacobian constitutes a conservative mapping.

\subsection{The definition of HS-Jacobian}

Assume that the polyhedral set $\mathcal{P}$ defined in \eqref{eq: polyhedral_set} is nonempty. Let $x\in \mathbb{R}^n$ be any given input vector, and $\Pi_{\mathcal{P}}(x)$ be the output of the projection layer defined in \eqref{eq: projector-general}. Therefore, it is required to compute the derivative of $\Pi_{\mathcal{P}}(\cdot)$ with respect to $x$ for backpropagation. Note that $\Pi_{\mathcal{P}}(\cdot)$ may not be differentiable \cite{facchinei2007finite}, but it is almost everywhere differentiable according to Rademacher's theorem \cite{rademacher1919partielle}. More specifically, the B-subdifferential and the Clarke Jacobian of $\Pi_{\mathcal{P}}(\cdot)$, denoted by $\partial_B\Pi_{\mathcal{P}}(\cdot)$ and $\partial_C\Pi_{\mathcal{P}}(\cdot)$ have been fully characterized \cite{robinson1987local, rockafellar1998variational, facchinei2007finite, mordukhovich2012variational}. However, finding an element in the B-subdifferential or the Clarke Jacobian of $\Pi_{\mathcal{P}}(\cdot)$ is computationally expensive, which is not suitable for backpropagation \cite{han1997newton}. Next, we introduce HS-Jacobian of $\Pi_{\mathcal{P}}(\cdot)$ as an efficiently computable replacement \cite{han1997newton}, which can be used for the backpropagation of the projection layer. It follows from the optimality condition corresponding to \eqref{eq: projector-general} that there exist $(\lambda,\mu) \in \mathbb{R}_+^m \times \mathbb{R}^l$ such that $(\Pi_{\mathcal{P}}(x), \lambda, \mu)$ satisfies the following KKT system
\begin{equation}
\label{eq: KKT_system}
    \left\{
    \begin{aligned}
        & \Pi_{\mathcal{P}}(x) - x + A^\top \lambda + B^\top \mu =0, \\
        & A \Pi_{\mathcal{P}}(x) -a \le 0, \quad B\Pi_{\mathcal{P}}(x) = b, \\
        & \lambda \ge 0, \quad \lambda^\top (A\Pi_{\mathcal{P}}(x) - a) = 0.
    \end{aligned}
    \right.
\end{equation}

For convenience, we denote the collection of multipliers at $x$ by
\begin{equation*}
    M(x) := \{(\lambda,\mu) ~|~ (x,\lambda,\mu) \text{ satisfies \eqref{eq: KKT_system}} \},
\end{equation*}
and index sets
\begin{equation}
\label{eq: active_set}
    I(x) := \{i \in \{1,2,\cdots,m\} ~|~ A_i \Pi_{\mathcal{P}}(x) = a_i\}, 
\end{equation}
\begin{equation*}
    \mathrm{supp}(\lambda) := \{i \in \{1,2,\cdots,m\} ~|~ \lambda_i > 0\},
\end{equation*}
where $A_i$ is the $i$th row of $A$ and $a_i, \lambda_i$ represent the $i$th elements of $a, \lambda$, respectively. Denote the collection of index sets by
\begin{equation*}
    \begin{aligned}
        \mathcal{E}(&x) := \{K\subseteq \{1,2,\cdots,m\} ~|~ \exists ~ (\lambda,\mu) \in M(x), \mathrm{s.t.} \\
        &\mathrm{supp}(\lambda) \subseteq K \subseteq I(x), 
         [A_K^\top ~ B^\top] \text{ is of full column rank}\},
    \end{aligned}
\end{equation*}
where $A_K$ represents the matrix consisting of rows indexed by $K$. Since $M(x)$ is a nonempty convex polyhedral set with no lines, there exists at least one extreme point of $M(x)$ \cite[Corollary 18.5.3]{rockafellar1970convex}, which further implies that the index set $\mathcal{E}(x)$ \cite{han1997newton} is non-empty. The HS-Jacobian \cite{han1997newton} of $\Pi_{\mathcal{P}}(\cdot)$ at $x$ is defined as 
\begin{equation}
\label{eq: HS-Jacobian-set}
\begin{aligned}
    \partial_{HS}\Pi_{\mathcal{P}}(x) :=
    \{ J_K \in \mathbb{R}^{n\times n} ~|~  K\in \mathcal{E}(x), 
    J_K = 
    I_n -  H_K\left( H_K^\top H_K\right)^{-1} H_K^\top
     \},
\end{aligned}
\end{equation}
where $H_K=[A_K^\top ~ B^\top]$.

\subsection{An efficient algorithm for finding an element of HS-Jacobian}

To incorporate the HS-Jacobian into end-to-end DNN training, it is required to compute one of its elements. In other words, we need to find an index set $K \in \mathcal{E}(x)$,  which may still be computationally expensive. Fortunately, Li et al \cite{li2020efficient} addressed this challenge to find an efficiently computable element in $\partial_{HS}\Pi_{\mathcal{P}}(x)$ as
\begin{equation}
\label{eq: HS_Jacobian_computation}
    J(x) := I_n - [A_{I(x)}^\top ~ B^\top] \left( \begin{bmatrix}A_{I(x)}\\ B\end{bmatrix} [A_{I(x)}^\top ~ B^\top]\right)^{\dagger} \begin{bmatrix} A_{I(x)}\\B \end{bmatrix},
\end{equation}
where $I(x)$ is the active set \eqref{eq: active_set} and $M^\dagger$ is the Moore-Penrose pseudo-inverse of matrix $M$. Importantly, this provides a computationally efficient way to obtain an element of $\partial_{HS}\Pi_{\mathcal{P}}(x)$ without the corresponding multipliers $(\lambda,\mu)$. Consequently, the computed element $J(x)$ of $\partial_{HS}\Pi_{\mathcal{P}}(x)$ can be directly employed for backpropagation, resulting in an HS-Jacobian based backpropagation. 

\subsection{The connections between HS-Jacobian and B-subdifferential}

It has already been mentioned in \cite{han1997newton} that the $\partial_B\Pi_{\mathcal{P}}(\cdot) \subseteq \partial_{HS}\Pi_{\mathcal{P}}(\cdot)$ and they are not identical in general. In the next theorem, we show that $\partial_{B}\Pi_{\mathcal{P}}(x) = \partial_{HS}\Pi_{\mathcal{P}}(x)$ if the LICQ holds at $\Pi_{\mathcal{P}}(x)$. The detailed proof can be found in section \ref{sec: appendix_thm_1}. 
\begin{theorem}
\label{thm: licq_identical}
    Let $\mathcal{P}$ be the polyhedral set defined in \eqref{eq: polyhedral_set}. Let $x \in \mathbb{R}^n$ be any given vector. Suppose the LICQ holds at $\Pi_{\mathcal{P}}(x)$, then $\partial_{B}\Pi_{\mathcal{P}}(x) = \partial_{HS}\Pi_{\mathcal{P}}(x)$. 
\end{theorem}
As a direct consequence, if there are no inequality constraints in $\mathcal{P}$, we know that the HS-Jacobian and B-subdifferential of $\Pi_{\mathcal{P}}(\cdot)$ are identical.

\subsection{The conservativity of HS-Jacobian}

Next, we prove that the HS-Jacobian $\partial_{HS}\Pi_{\mathcal{P}}$ defined in \eqref{eq: HS-Jacobian-set} is a conservative mapping for the projection operator $\Pi_{\mathcal{P}}$. Our proof is inspired by the equivalence between semismoothness and conservativity for semi-algebraic mappings\footnote{See section \ref{sec: appendix_pre} for the definition of semi-algebraic mapping.} \cite{davis2022conservative}. To this end, we first show that both the projection operator $\Pi_{\mathcal{P}}$ and its HS-Jacobian $\partial_{HS}\Pi_{\mathcal{P}}$ are semi-algebraic mappings. We then prove the conservativity of $\partial_{HS}\Pi_{\mathcal{P}}$ by invoking this equivalence. These results are formalized in the following theorem. Detailed proofs are provided in section \ref{sec: appendix_thm_2}.

\begin{theorem}
\label{thm: semi_algebraic_Jacobian}
    The HS-Jacobian $\partial_{HS}\Pi_{\mathcal{P}}: \mathbb{R}^n \rightrightarrows \mathbb{R}^{n \times n}$ is a semi-algebraic mapping. Moreover, it is a conservative mapping for the projection operator $\Pi_{\mathcal{P}}: \mathbb{R}^n \to \mathbb{R}^n$.
\end{theorem}

This theorem provides a theoretical foundation to use the element \eqref{eq: HS_Jacobian_computation} from HS-Jacobian for backpropagation. Next, we move to introduce the efficient and provable algorithms for training of linearly constrained DNNs.

\section{An HS-Jacobian based backpropagation algorithm for training linearly constrained DNNs}

In this section, we demonstrate that the Adam \cite{kingma2015adam} with the proposed HS-Jacobian based backpropagation is an efficient and provable algorithm for training DNNs with linear constraints.

\subsection{The decomposition of a DNN in terms of parameters}

For brevity, we denote all the trainable model parameters by 
\begin{equation*}
    \theta=(\theta_1,\theta_2,\cdots,\theta_p) \in \mathbb{R}^p,
\end{equation*}
and let $f_\theta: \mathbb{R}^{n_1} \to \mathbb{R}^{n_2}$ denote the DNN incorporating projection layers. Given an input $x \in \mathbb{R}^{n_1}$ and its corresponding ground truth $y \in \mathbb{R}^{n_2}$, we define function $\Phi: \mathbb{R}^p \to \mathbb{R}$ with respect to $\theta$ as 
\begin{equation*}
    \Phi(\theta;x,y) := \ell(y, f_\theta(x)),
\end{equation*}
where $\ell: \mathbb{R}^{n_2} \times \mathbb{R}^{n_2} \to \mathbb{R}$ is some given loss function.

Following the framework of \cite{bolte2021conservative}, we can express $\Phi(\theta;x,y)$ as a composition of elementary functions by unfolding the DNN via a directed acyclic computational graph. We sequentially define the auxiliary variables $\theta_{p+1}, \cdots, \theta_q \in \mathbb{R}$ along the forward pass of the DNN. Specifically, for the auxiliary variable $\theta_k$, it is defined by an elementary function $\phi_k$ in terms of a subset of previously defined variables, i.e.,
\begin{equation*}
    \theta_k:=\phi_k(\theta_{\mathrm{parents(k)}}), \quad \theta_{\mathrm{parents}(k)} := (\theta_i)_{i \in \mathrm{parents}(k)},
\end{equation*}
where $\mathrm{parents}(k) \subset \{1,2,\cdots,k-1\}$ is the set of indices on which $\theta_k$ directly depends and $k=p+1,\cdots,q$. The function $\phi_k$ corresponds to a coordinate of elementary mappings such as affine mappings, activation mappings, loss functions, or the projection operators\footnote{Although the projection operator $\Pi_{\mathcal{P}}$ may not have a closed formula, we still treat it as an elementary function here.}. We illustrate the corresponding computational graph by Fig.~\ref{fig: decomposition_illustration}. The definition procedure is shown in Algorithm \ref{alg: definition_program}. 

\begin{figure}[thbp]
    \centering
    \includegraphics[width=0.7\linewidth]{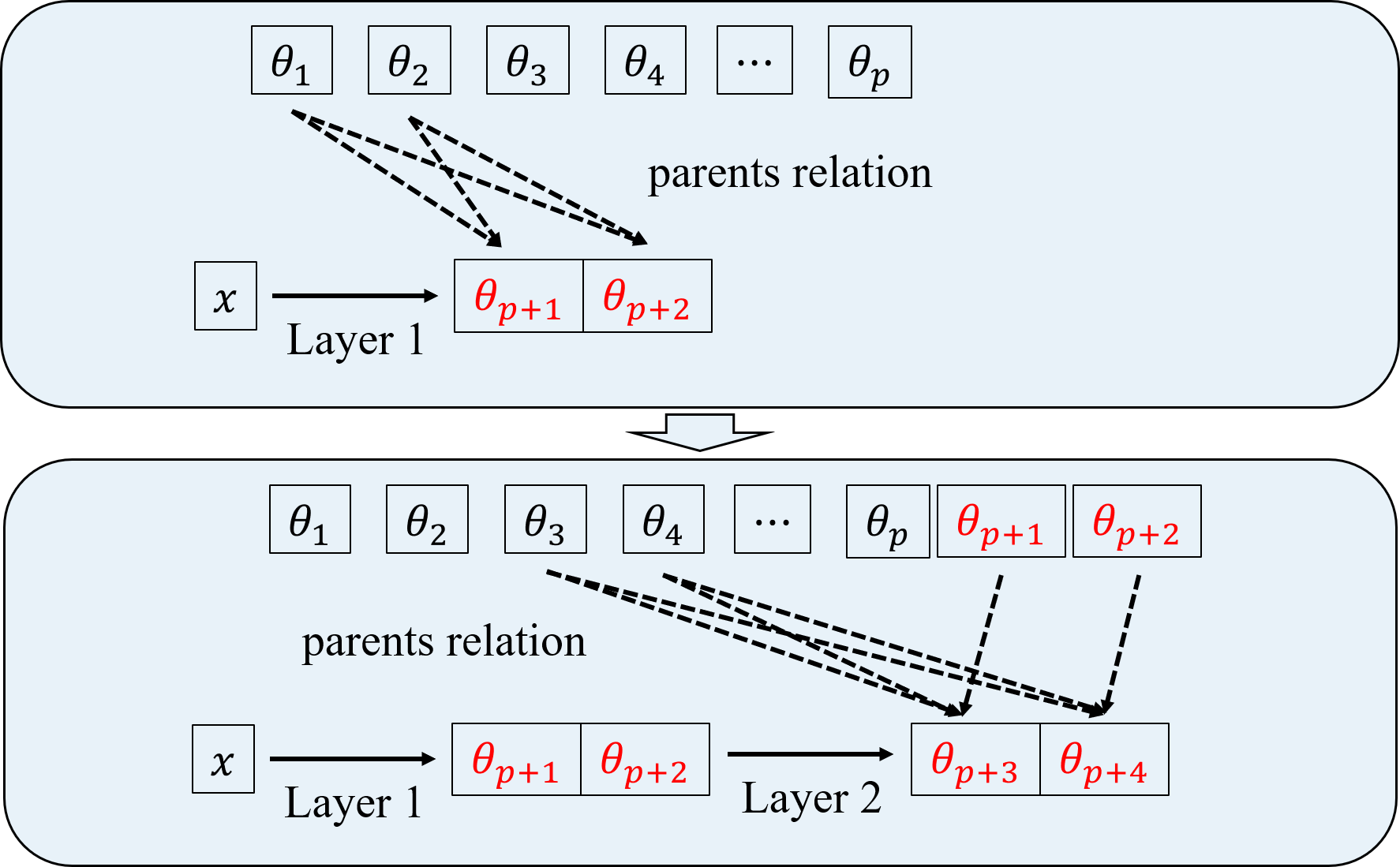}
    \caption{A simple illustration of the directed acyclic computational graph corresponding to the definition procedure. The model parameters $\theta_1,\cdots,\theta_p$ and the auxiliary variables $\theta_{p+1},\theta_{p+2}, \theta_{p+3}, \theta_{p+4}$ refer to the nodes. The dashed lines refer to the directed edges (corresponding to ``parents relation''). The auxiliary variables and their ``parents'' are determined sequentially along the forward pass. In this graph, $\mathrm{parents}(p+1)=\{1,2\}, \mathrm{parents}(p+2)=\{1,2\}, \mathrm{parents}(p+3)=\{3,4,p+1\}, \mathrm{parents}(p+4)=\{3,4,p+2\}$}
    \label{fig: decomposition_illustration}
\end{figure}

\begin{algorithm}[!h]
\caption{Definition program of $\Phi: \mathbb{R}^p \to \mathbb{R}$}
\label{alg: definition_program}
\renewcommand{\algorithmicrequire}{\textbf{Input:}}
\renewcommand{\algorithmicensure}{\textbf{Output:}}
\begin{algorithmic}[1]
\REQUIRE Initial parameters $\theta=(\theta_1,\theta_2,\cdots,\theta_p)$, given data $(x,y)$.
\FOR{$k = p+1,p+2,\cdots, q$}
    \STATE Specify $\theta_k=\phi_k(\theta_{\mathrm{parents}(k)}),$
    where $\theta_{\mathrm{parents}(k)} = (\theta_i)_{i \in \mathrm{parents}(k)}$.
\ENDFOR
\ENSURE $\theta_{p+1}, \cdots, \theta_q$ and $\phi_{p+1}, \cdots, \phi_q$.
\end{algorithmic}
\end{algorithm}

In applications, the elementary functions $\phi_k$ in Algorithm \ref{alg: definition_program} are determined automatically during the construction of the DNN $f_\theta$. To illustrate Algorithm \ref{alg: definition_program}, we present a simple example below.

\begin{example}
    Suppose input $x\in \mathbb{R}^2$, parameters $W\in \mathbb{R}^{2\times 2}, \beta\in \mathbb{R}^2$. Let
\begin{equation*}
    x=\begin{pmatrix}
        x_1 \\ x_2
    \end{pmatrix}, \quad W=\begin{pmatrix}
        w_{11} ~ w_{12} \\ w_{21} ~ w_{22}
    \end{pmatrix}, \quad \beta = \begin{pmatrix}
        \beta_1 \\ \beta_2
    \end{pmatrix}.
\end{equation*}
Suppose $\theta_1=w_{11}, \theta_2=w_{21}, \theta_3=w_{12}, \theta_4=w_{22}, \theta_5=\beta_1, \theta_6=\beta_2$ and $\theta=(\theta_1,\cdots, \theta_6)^\top \in \mathbb{R}^6$. Let
\begin{equation*}
    f_\theta(x) = (\Pi_{\mathcal{P}}(Wx+\beta))^\top \Pi_{\mathcal{P}}(Wx+\beta), \quad \ell(y_1,y_2) = |y_1 - y_2|,
\end{equation*}
where $\mathcal{P}=\{x\in \mathbb{R}^2 ~|~ x_1+x_2=1, x_1 \ge 0.3\}$. Then for input $x\in \mathbb{R}^2$ and ground truth $y\in \mathbb{R}$, the parameter function $\Phi(\theta;x,y)=\ell(y,f_\theta(x))$ can be recursively defined by
\begin{equation*}
    \begin{aligned}
        \begin{pmatrix}
        \theta_7 \\ \theta_8
    \end{pmatrix} = \begin{pmatrix}
        \theta_1 x_1 + \theta_3 x_2 + \theta_5 \\ \theta_2x_1+\theta_4 x_2 + \theta_6
    \end{pmatrix}, \quad \begin{pmatrix}
        \theta_9 \\ \theta_{10}
    \end{pmatrix} = \Pi_{\mathcal{P}}\left( \begin{pmatrix}
        \theta_7 \\ \theta_8
    \end{pmatrix} \right), 
    \end{aligned}
\end{equation*}
\begin{equation*}
    \theta_{11} = \theta_9^2 +\theta_{10}^2, \quad \theta_{12} = |\theta_{11} - y|.
\end{equation*}
In this case, $\mathrm{parents}(7)=\{1,3,5\}, \mathrm{parents}(8)=\{2,4,6\}, \mathrm{parents}(9)=\{7,8\}, \mathrm{parents}(10)=\{7,8\}$, $\mathrm{parents}(11)=\{9,10\}, \mathrm{parents}(12)=\{11\}$.
\end{example}

As revealed in \cite{bolte2021conservative, xiao2024adam}, definability in an $o$-minimal structure \cite{pillay1984definable, van1996geometric} plays a crucial role in the convergence analysis of DNN training. It is well known that the semi-algebraic mappings are definable in the $o$-minimal structure $\mathbb{R}_{\mathrm{alg}} = (\mathbb{R}, +,\cdot,<,0,1)$ \cite{bierstone1988semianalytic}. Moreover, as a direct consequence of Khovanskii's Theorem \cite{khovanskii1980class} and Wilkie's Theorem \cite{wilkie1996model}, the structure $\mathbb{R}_{\mathrm{exp}} = (\mathbb{R},+, \cdot, \mathrm{exp}, <, 0, 1)$ is also $o$-minimal \cite{marker1996model}, in which the semi-algebraic sets are definable as well. 

When constructing a DNN, a variety of activation functions may be employed, such as ReLU \cite{nair2010rectified}, Leaky-ReLU \cite{maas2013rectifier}, Sigmoid \cite{rumelhart1986learning}, hyperbolic tangent \cite{lecun1989backpropagation} and so on. It is known that the first two are semi-algebraic, while the latter two are definable in $\mathbb{R}_{\mathrm{exp}}$ \cite{bolte2021conservative}. In addition, the common loss functions such as $\ell_1$-loss, mean squared error (MSE) loss, hinge loss, logistic loss, and cross-entropy loss are also definable in $\mathbb{R}_{\mathrm{exp}}$ \cite{xiao2024adam}. Moreover, the projection layer is also definable in $\mathbb{R}_{\mathrm{exp}}$ since it is a semi-algebraic mapping (see section \ref{sec: appendix_thm_2}). For simplicity, ``definable'' will henceforth refer to ``definable in $\mathbb{R}_{\mathrm{exp}}$'' throughout the remainder of this paper. Based on the preceding discussion, we make the following mild assumption.

\begin{assumption}
\label{assum: block_definable}
    All elementary functions $\phi_k$ specified by Algorithm \ref{alg: definition_program} for defining  $\Phi$ are locally Lipschitz continuous and definable.
\end{assumption}

\subsection{A nonsmooth automatic differentiation algorithm for linearly constrained DNNs}

When training DNNs, automatic differentiation (AD) \cite{griewank2008evaluating, baydin2018automatic} plays a central role in computing parameter updates during backpropagation. The classical AD algorithm is designed for smooth functions. However, many activation functions and layers in modern DNNs are nonsmooth, in which case the classical procedure fails to produce valid gradients. To address this limitation, Bolte and Pauwels \cite{bolte2021conservative} proposed a nonsmooth AD framework, where the gradients of the elementary functions $\phi_k$ are replaced by their corresponding conservative fields. Under this framework, the reverse mode nonsmooth AD algorithm for computing the ``gradient'' of a DNN is summarized in Algorithm \ref{alg: reverse_mode_AD}.

\begin{algorithm}[!h]
\caption{Reverse mode of nonsmooth AD algorithm for $\Phi: \mathbb{R}^p \to \mathbb{R}$}
\label{alg: reverse_mode_AD}
\renewcommand{\algorithmicrequire}{\textbf{Input:}}
\renewcommand{\algorithmicensure}{\textbf{Output:}}
\begin{algorithmic}[1]
\REQUIRE All involving parameters $\theta=(\theta_1,\theta_2,\cdots,\theta_q)$, given data $(x,y)$, a conservative field $D_k: \mathbb{R}^{|\mathrm{parents}(k)|} \rightrightarrows \mathbb{R}^{|\mathrm{parents}(k)|}$ for each $\phi_k, k=p+1,\cdots,q$. 
\STATE Initialization: $g=(0,0,\cdots,0,1) \in \mathbb{R}^q$.
\FOR{$k = q,q-1,\cdots,p+1$}
\STATE $d_k=(d_{kj})_{j=1}^{\mathrm{parents}(k)} \in D_k(\theta_{\mathrm{parents}(k)})$.
\FOR{$j\in \mathrm{parents}(k)$}
    \STATE $g_j=g_j +g_k d_{kj}$.
\ENDFOR
\ENDFOR
\ENSURE $g=(g_1,g_2,\cdots,g_p) \in \mathbb{R}^p$.
\end{algorithmic}
\end{algorithm}
\begin{remark}
    The existence of conservative field $D_k$ for $\phi_k$ is based on Assumption \ref{assum: block_definable} and Proposition 2 of \cite{bolte2021conservative}.
\end{remark}

Since we have established that the HS-Jacobian is a conservative mapping for the projection operator, it can be seamlessly integrated into the nonsmooth AD algorithm. Moreover, given $(x,y)$, for each $(\theta_1,\cdots,\theta_p)$, the outputs of Algorithm \ref{alg: reverse_mode_AD} form a so-called AD field $\mathcal{D}_{\Phi(\cdot;x,y)}: \mathbb{R}^p \rightrightarrows \mathbb{R}^p$,
\begin{equation*}
\begin{aligned}
    \mathcal{D}_{\Phi(\cdot;x,y)}(\theta) := \{g\in \mathbb{R}^p ~|~ \text{all possible outputs} \\ 
    \text{ of Algorithm \ref{alg: reverse_mode_AD} given }(x,y) \}.
\end{aligned}
\end{equation*}
Without ambiguity, we abbreviate $\mathcal{D}_{\Phi(\cdot;x,y)}$ for $\mathcal{D}_{\Phi}(\cdot)$. By Theorem 8 of \cite{bolte2021conservative}, the AD field $\mathcal{D}_{\Phi}$ is a conservative field for the parameter function $\Phi$ defined by Algorithm \ref{alg: definition_program}. To analyze the convergence of Adam optimizer for linearly constrained DNNs \cite{kingma2015adam}, we further specify the selection of conservative mappings in Algorithm \ref{alg: reverse_mode_AD}. 

First, consider the elementary functions $\phi_k$ in Algorithm \ref{alg: definition_program} that correspond to the projection operator $\Pi_{\mathcal{P}}: \mathbb{R}^n \to \mathbb{R}^n$. Without loss of generality, suppose $\phi_{k+1}, \cdots, \phi_{k+n}$ are the functions associated with this operator. We set $D_{k+1}, \cdots, D_{k+n}$ to be the rows of the HS-Jacobian $\partial_{HS}\Pi_{\mathcal{P}}$, which is a conservative mapping by Theorem \ref{thm: semi_algebraic_Jacobian}. By Lemma 4 of \cite{bolte2021conservative}, each row $D_{k+i}$ remains a conservative field for the corresponding $\phi_{k+i}~(i=1,\cdots,n)$. Furthermore, by Theorem \ref{thm: semi_algebraic_Jacobian} and the Tarski--Seidenberg theorem (see Corollary \ref{corol: projection_semialgebraic} in section \ref{sec: appendix_pre}), the conservative fields $D_{k+1},\cdots, D_{k+n}$ are semi-algebraic.

For other elementary functions, we adopt the following assumption on the choice of their conservative fields $D_k$, which is also adopted in \cite{bolte2021conservative, xiao2024adam}.

\begin{assumption}
\label{assum: mapping_is_clarke_subdifferential}
    The conservative field $D_k$ for function $\phi_k$ is taken to be its Clarke subdifferential whenever $\phi_k$ is not a coordinate of a projection operator.
\end{assumption}

The reasons for choosing the Clarke subdifferential of $\phi_k$ as its conservative mapping $D_k$ are multi-fold. First, by Proposition 4 of \ref{sec: appendix_pre}, if a function admits a conservative field, then its Clarke subdifferential would be its conservative mapping. The second reason is that the Clarke subdifferential of a definable function remains definable \cite{bolte2007clarke}. Moreover, in contrast to the difficulty of finding an element in the Clarke Jacobian of $\Pi_{\mathcal{P}}(\cdot)$, it is computationally efficient to find an element in the Clarke Jacobian of other elementary functions in DNNs. With this selection rule, we can ensure all the conservative fields $D_k$ in Algorithm \ref{alg: reverse_mode_AD} are definable.

A fundamental result in model theory \cite{marker2002model} states that definability is preserved under finite composition \cite[Lemma 6.1 and Remark 6.7]{denkowski2017definable}, multiplication, and summation \cite{xiao2024adam}. Thus, under Assumption \ref{assum: block_definable}, the function $\Phi(\theta;x,y)$ defined by Algorithm \ref{alg: definition_program} is definable, as it is a finite composition of definable functions \cite{davis2020stochastic}. Moreover, since all selected $D_k$ related to $\phi_k$ are definable, the AD field $\mathcal{D}_{\Phi}$ generated via summation and multiplication according to Algorithm \ref{alg: reverse_mode_AD} is also definable \cite{bolte2021conservative, xiao2024adam}.

\subsection{The convergence of training DNNs by HS-Jacobian based backpropagation}

We consider training a DNN equipped with projection layers on a dataset $\Omega=\{(x^{(t)}, y^{(t)})\}_{t=1}^N$. For each data point $(x^{(t)}, y^{(t)}) \in \Omega$, define $\varphi_t(\theta)=\Phi(\theta;x^{(t)}, y^{(t)})$. The optimization problem for training the DNN can be expressed as 
\begin{equation}
\label{eq: empirical_objective_function}
    \min_{\theta \in \mathbb{R}^p} \quad \varphi(\theta) :=\frac{1}{N} \sum_{t=1}^N \varphi_t(\theta).
\end{equation}
This corresponds to minimizing the expectation over the probability space $(\Omega, \mathcal{F}, \mathbb{P})$, where $\mathcal{F}=2^{\Omega}$ is the $\sigma$-algebra and $\mathbb{P}$ is the discrete uniform probability. Based on the AD field $\mathcal{D}_{\Phi}$, define a set-valued mapping
\begin{equation*}
    \mathcal{D}: \mathbb{R}^p \times \Omega \rightrightarrows \mathbb{R}^p, \quad (\theta,x,y) \mapsto \mathcal{D}_{\Phi(\cdot;x,y)}(\theta).
\end{equation*}

We further impose the following mild assumptions commonly used in the convergence analysis of Adam-type algorithms \cite{xiao2024adam}.

\begin{assumption}
\label{assum: conservative_mapping_setting}
    \begin{enumerate}
        \item There exists a measurable selection $\chi: \mathbb{R}^p \times \Omega \to \mathbb{R}^p$ such that $\chi(\theta,x,y) \in \mathcal{D}(\theta,x,y)$ for all $\theta \in \mathbb{R}^p$ and $(x,y) \in \Omega$.
        \item There exists a constant $M_{\Omega} > 0$ such that 
        \begin{equation*}
            \sup_{d \in \mathcal{D}(\theta,x,y)} \| d\| \le M_{\Omega}
        \end{equation*}
        for all $\theta \in \mathbb{R}^p$ and $(x,y) \in \Omega$.
    \end{enumerate}
\end{assumption}

The details of the Adam optimizer for \eqref{eq: empirical_objective_function} are described by the following algorithm. Note that the vector exponentiation and $\odot$ are element-wise operators in this algorithm.

\begin{algorithm}[!h]
\caption{Adam optimizer for \eqref{eq: empirical_objective_function}}
\label{alg: Adam_optimizer}
\renewcommand{\algorithmicrequire}{\textbf{Input:}}
\renewcommand{\algorithmicensure}{\textbf{Output:}}
\begin{algorithmic}[1]
\REQUIRE Initial points $\theta_0\in \mathbb{R}^p, m_0\in \mathbb{R}^p, v_0\in \mathbb{R}^p$, parameters $\tau_1,\tau_2, \varepsilon >0$. 
\STATE Set $t=0$.
\WHILE{not satisfied termination condition}
\STATE Independently choose $(x^{(t)}, y^{(t)}) \in \Omega$.
\STATE Compute $g_t=\chi(\theta_t, x^{(t)}, y^{(t)}) \in \mathcal{D}(\theta_t,x^{(t)}, y^{(t)})$.
\STATE Choose the stepsize $\eta_t$.
\STATE $m_{t+1} = (1-\tau_1\eta_t)m_t + \tau_1 \eta_t g_t.$
\STATE $v_{t+1} = (1-\tau_2\eta_t) v_t + \tau_2\eta_t g_t\odot g_t.$
\STATE Compute the scaling parameters $\rho_{m,t+1}$ and $\rho_{v,t+1}$.
\STATE $\theta_{t+1} = \theta_t-\eta_t (\rho_{v,t+1} |v_{t+1}| +\varepsilon) ^{-\frac{1}{2}} \odot \rho_{m,t+1} m_{t+1}.$
\STATE $t=t+1$.
\ENDWHILE
\ENSURE $\theta_t \in \mathbb{R}^p$.
\end{algorithmic}
\end{algorithm}

In the convergence analysis of Algorithm \ref{alg: Adam_optimizer}, we impose the following mild assumptions on the hyperparameter settings and the sequence boundedness. These assumptions are standard in the literature and have also been adopted in \cite{xiao2024adam}.

\begin{assumption}
\label{assum: adam_algorithm_setting}
    \begin{enumerate}
        \item The sequence $\{\theta_t\}$ is almost surely bounded.
        \item The sequence of step sizes $\{\eta_t\}$ satisfies
        \begin{equation*}
            \eta_t >0, \quad \sum_{t=0}^{+\infty} \eta_t=+\infty, \quad \lim_{t \to +\infty} \eta_t \log(t) = 0.
        \end{equation*}
        \item The scaling parameters $\{\rho_{m,t}\}$ and $\{\rho_{v,t}\}$ satisfy
        \begin{equation*}
            \lim_{t \to +\infty} \rho_{m,t} = 1, \quad \lim_{t \to +\infty } \rho_{v,t} = 1.
        \end{equation*}
    \end{enumerate}
\end{assumption}

Based on Corollary 1 of \cite{xiao2024adam}, we have the convergence of Algorithm \ref{alg: Adam_optimizer}.
\begin{theorem}
\label{thm: adam_convergence}
    For any sequence $\{\theta_t\}$ generated by Algorithm \ref{alg: Adam_optimizer} with parameters satisfying $\tau_2 \le 4\tau_1$, if Assumption \ref{assum: block_definable}, \ref{assum: mapping_is_clarke_subdifferential}, \ref{assum: conservative_mapping_setting}, \ref{assum: adam_algorithm_setting} hold, then almost surely, every cluster point of $\{\theta_t\}$ is a $\mathcal{D}_{\varphi}$-critical point of the loss function $\varphi$ \eqref{eq: empirical_objective_function} and the sequence $\{\varphi(\theta_t)\}$ converges, where $\mathcal{D}_{\varphi}$ is the  convex conservative mapping for $\varphi$
    \begin{equation}
    \label{eq: conv_of_conservative_for_loss}
        \mathcal{D}_{\varphi}: \mathbb{R}^p \rightrightarrows \mathbb{R}^p, \quad \theta \mapsto \frac{1}{N} \mathrm{conv}\left(\sum_{t=1}^N \mathcal{D}_{\varphi_t}(\theta) \right).
    \end{equation}
\end{theorem}

Detailed proofs are provided in section \ref{sec: appendix_thm_3}. This result demonstrates that the proposed HS-Jacobian based backpropagation method preserves the convergence properties of Adam when training DNNs with linear constraints, thereby providing a rigorous theoretical guarantee for its deployment in applications.

\section{Numerical experiments}

In this section, we test the performance of the proposed algorithms for training linearly constrained DNNs through three real applications in finance, computer vision, and network architecture design. The numerical results will demonstrate our proposed HS-Jacobian based algorithm can achieve superior performance for training linearly constrained DNNs in terms of feasibility, training loss, model performance, and memory cost, compared to other popular methods \cite{wang2023linsatnet, xie2025mhc}.

\subsection{Portfolio allocation}

Portfolio allocation is one of the most important problems in finance. Some basic constraints in asset allocation include that each asset is assigned a nonnegative weight and the sum of all weights equals one \cite{jagannathan2003risk}. In practice, it is often beneficial to introduce expert preference to ensure sufficient holdings in promising assets since those assets may yield high long-term returns. For a predicted allocation weight vector $w=(w_1,\cdots,w_n)^\top\in \mathbb{R}^n$, the constraints on $w$ can then be formulated as
\begin{equation}
\label{eq: cons_portfolio}
    \sum_{i=1}^n w_i = 1, \quad \sum_{i \in \mathcal{C}}w_i \ge \delta, \quad w\ge 0,
\end{equation}
where $\mathcal{C}$ is a collection of assets chosen by experts and $\delta$ is the minimum total weight allocated to those assets. 

The StemGNN \cite{Cao2020spectral} is widely adopted as the backbone network to jointly predict future asset returns and portfolio weights in DNN based portfolio allocation, where the linear constraints \eqref{eq: cons_portfolio} are imposed on the predicted portfolio weights. The loss function is defined as a weighted sum of the prediction error on future asset returns and the Sharp ratio \cite{sharpe1966mutual, sharpe1994sharpe}. We compare our algorithm, which trains this linearly constrained DNN model directly, to a representative unrolling-based method LinSATNet \cite{wang2023linsatnet}.  The dataset consists of daily prices of 493 assets from S\&P 500 index (i.e. $n=493$).

\begin{figure}[htbp]
    \centering
    \includegraphics[width=0.8\linewidth]{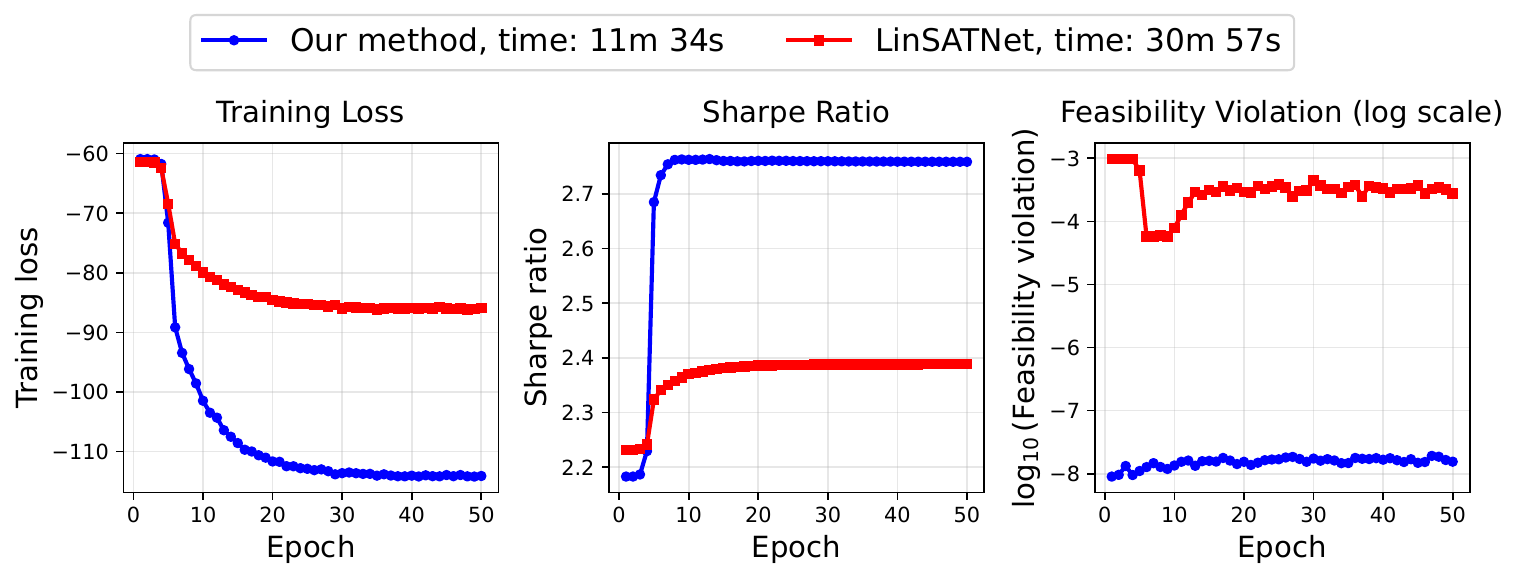}
    \caption{The comparison of our method and LinSATNet in portfolio allocation. The left figure shows the average training loss over each epoch. The middle figure shows the average Sharpe ratio on the test set. The right figure shows the average base-10 logarithm of feasibility violation during training.}
    \label{fig: portfolio_all}
\end{figure}

The numerical results are presented by Fig.~\ref{fig: portfolio_all}. These results show that our method trained with HS-Jacobian based backpropagation consistently achieves superior performance, lower training loss, and significantly smaller feasibility violations compared with LinSATNet. Besides, it also requires less training and evaluation time. More experimental settings are detailed in section \ref{sec: appendix_portfolio}. 

\subsection{Partial graph matching}

Partial graph matching considers the scenario where only a subset of nodes in one graph can be matched to a subset of nodes in another graph \cite{wang2021neural}. Let $\mathcal{G}_1=(\mathcal{V}_1,\mathcal{E}_1), \mathcal{G}_2=(\mathcal{V}_2, \mathcal{E}_2)$ be two graphs with $|\mathcal{V}_1|=d_1, |\mathcal{V}_2|=d_2$. Denote the matching score between node $i$ in $\mathcal{G}_1$ and node $j$ in $\mathcal{G}_2$ by $X_{ij} \in [0,1]$. Then the resulting matching score matrix $X=(X_{ij})_{d_1 \times d_2}$ needs to satisfy the following constraints
\begin{equation}
\label{eq: cons_graph}
    \begin{aligned}
        & \sum_{j=1}^{d_2} X_{ij} \le 1, \forall i, \quad \sum_{i=1}^{d_1} X_{ij} \le 1, \forall j, \\
        & \sum_{i=1}^{d_1} \sum_{j=1}^{d_2} X_{ij} \le \alpha, \quad X \ge 0,
    \end{aligned}
\end{equation}
where $\alpha $ is a prescribed upper bound on the number of matched node pairs.

In the experiments, we employ the network used in \cite{wang2023linsatnet} to predict the matching scores of two graphs with the linear constraints specified in \eqref{eq: cons_graph}. The training loss is the binary cross-entropy between the ground truth permutation matrix and $X$. Again, we compare our algorithm, which trains this linearly constrained DNN model directly, to a representative unrolling-based method LinSATNet \cite{wang2023linsatnet}. The experiments are conducted on the Pascal VOC Keypoint dataset \cite{everingham2010pascal} with Berkeley annotations \cite{bourdev2009poselets}, following the same setting as in \cite{rolinek2020deep, wang2023linsatnet}. We train the whole model from scratch. During training, we set $\alpha$ equal to the ground truth number of matched nodes. When evaluating the model performance, the Hungarian algorithm \cite{kuhn1955hungarian} is applied to the predicted matching score matrix $X$ to obtain a permutation matrix, from which the F1 score is computed based on the ground truth permutation matrix.

\begin{figure}[htbp]
    \centering
    \includegraphics[width=0.8\linewidth]{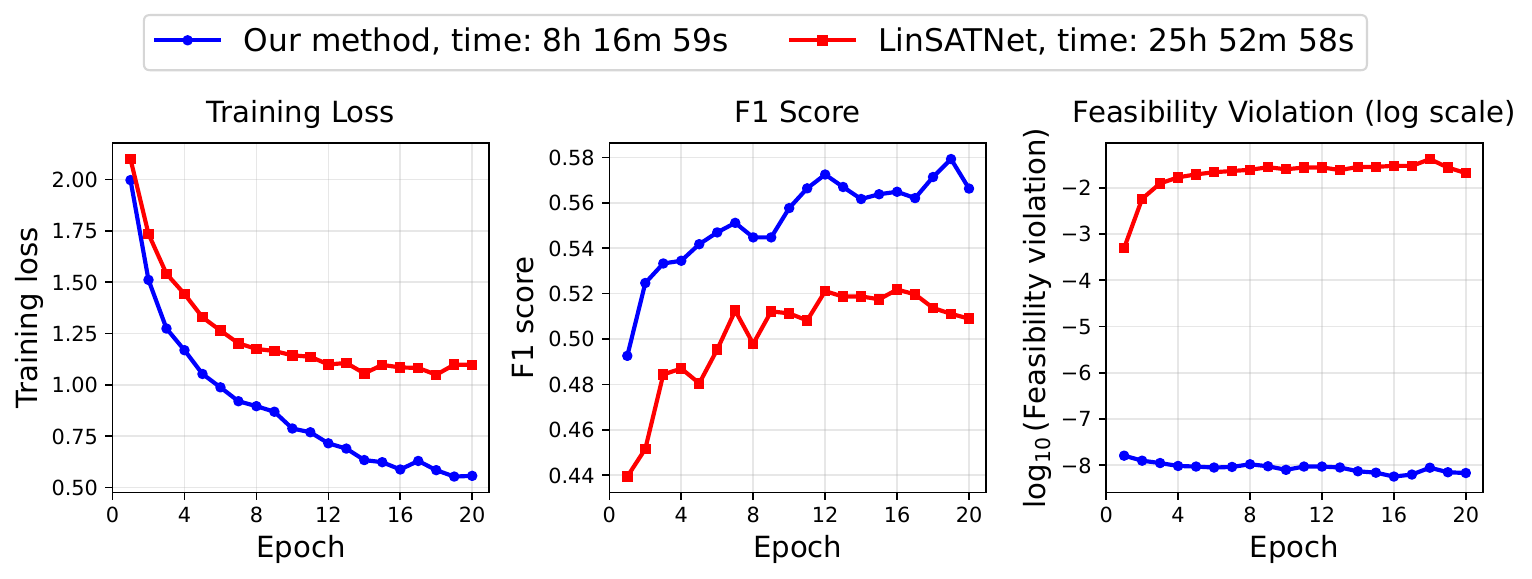}
    \caption{The comparison of our method (trained with the HS-Jacobian based backpropagation) and LinSATNet in graph matching. The left figure represents the average training loss over each epoch. The middle figure represents the average F1 score on test set. The right figure represents the average base-10 logarithm of feasibility violation during training. We also display the total training and evaluation time on the top.}
    \label{fig: graph_all}
\end{figure}

\begin{figure*}[tbp]
    \centering
    \includegraphics[width=1\textwidth]{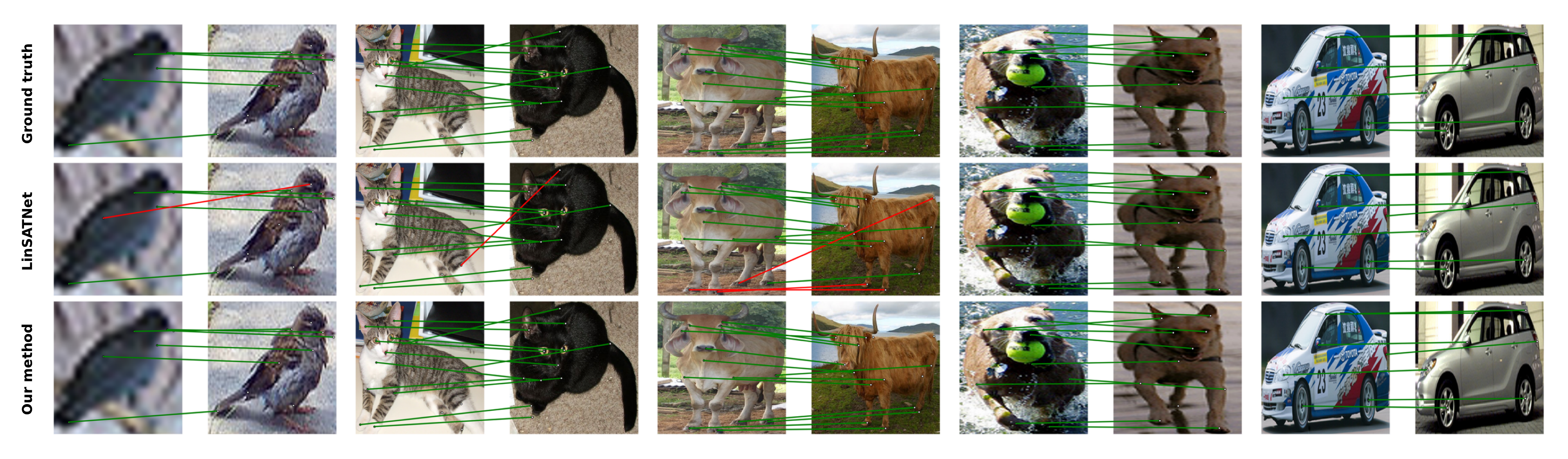}
    \caption{Visualization of some graph matching results by LinSATNet and our method on Pascal VOC Keypoint dataset. The green lines represent correct matching. The red lines represent incorrect matching. The images in the same column are chosen from the same category. The categories are bird, cat, cow, dog, and car.}
    \label{fig: graph_visual}
\end{figure*}

As shown in Fig.~\ref{fig: graph_all}, compared to LinSATNet, our proposed method can achieve a faster loss decrease with much less training time. It also delivers superior model performance and maintains significantly higher feasibility. Some selected graph matching results are visualized in Fig.~\ref{fig: graph_visual}. More experimental settings and visualization results are detailed in section \ref{sec: appendix_graph}. 

\subsection{Manifold-constrained hyper-connection}

The recently introduced Hyper-Connections (HC) architecture~\cite{zhu2024hyper} generalizes the deep residual network~\cite{he2016deep} and achieves success in many applications. Specifically, HC employs a hyper-parameter $c$ to expand the input $x \in \mathbb{R}^{d}$ into a multi-path representation $X = \mathbf{1}_c x^\top \in \mathbb{R}^{c \times d}$. The HC then modulate the information flow across these $c$ paths to obtain the output $y \in \mathbb{R}^{d}$ as
\begin{equation*}
\label{eq:hyperconn}
y^\top = \mathbf{1}^{\top}_{c} \left( H^{\mathrm{res}} X + H^{\mathrm{post}} \mathcal{F} \left( \left( H^{\mathrm{pre}} X\right)^{\top}, \theta \right)^\top \right),
\end{equation*}
where $H^{\mathrm{res}} \in \mathbb{R}^{c \times c}$, $H^{\mathrm{pre}} \in \mathbb{R}^{1 \times c}$, and $H^{\mathrm{post}} \in \mathbb{R}^{c \times 1}$ are learnable weights. 

Despite its flexibility, it has been observed to be unstable during the training of Large Language Models with the HC architecture~\cite{xie2025mhc}. To address this, Xie et al.~\cite{xie2025mhc} proposed the Manifold-constrained HC (mHC) by constraining $H^{\mathrm{res}}$ in the Birkhoff polytope
\begin{equation*}
\mathcal{P}_B=\{H\in \mathbb{R}^{c \times c}~|~ H\mathbf{1}_{c} = \mathbf{1}_{c}, H^{\top}\mathbf{1}_{c} = \mathbf{1}_{c}, H \ge 0\}.
\end{equation*}
The mHC architecture has been used in the very recent released DeepSeek V4 model~\cite{deepseekai2026deepseekv4}, which has achieved comparable performance to the GPT 5.4 model across multiple benchmark tests. 

To evaluate the effectiveness of our method relative to the Sinkhorn algorithm used in mHC~\cite{xie2025mhc}, we conduct image classification experiments using a Vision Transformer (ViT)~\cite{dosovitskiy2021an} with $c=8$ on the ImageNet dataset~\cite{deng2009imagenet}. As illustrated in Fig.~\ref{fig: vit_train_loss}, our method consistently maintains a slightly lower loss trajectory. We attribute the subtle nature of this gap to the fact that both methods successfully stabilize gradient flows by approximating the Birkhoff polytope, which is the primary driver of training stability. Nevertheless, the persistent improvement achieved by our method highlights the benefits of eliminating approximation errors, ensuring a more precise adherence to the manifold throughout the optimization process. This improvement is further evidenced by the Top-1 validation accuracy, as illustrated in Fig.~\ref{fig: vit_valid_acc}. Although the absolute performance margin is constrained by the model's capacity and the condensed training schedule, our proposed method consistently establishes a superior performance compared to the Sinkhorn baseline.

Beyond optimization precision, our method offers significant memory efficiency as demonstrated in Table~\ref{tab:memory}. The memory surge in the Sinkhorn approach is primarily attributed to the necessity of caching intermediate variables across multiple iterations to facilitate backpropagation. In contrast, our method circumvents this overhead, allowing for larger batch sizes without resorting to gradient checkpointing~\cite{chen2016training}, a technique that typically compromises training speed to save memory. This efficiency enables our method to handle larger-scale data within the same hardware constraints. 
\begin{table}[t]
    \caption{Peak memory consumption of different methods.}
    \centering
    \begin{tabular}{cccc}
    \toprule
    Train Batchsize & Method & Sinkhorn Iteration & Memory (MB) \\
    \midrule
    \multirow{3}{*}{128} & Our Method &  $\backslash$   &  64,190 \\
    \cline{2-4}
    & \multirow{2}{*}{Sinkhorn} & 20 &  72,246   \\ 
    &                           & 30 &  75,786   \\
    \bottomrule
    \end{tabular}
    \label{tab:memory}
\end{table}

Comprehensive hyperparameters for the ViT model configuration and training are detailed in section \ref{sec: appendix_mhc}. For the results shown in Fig.~\ref{fig: vit_train_loss} and \ref{fig: vit_valid_acc}, we set the Sinkhorn iteration to 20, consistent with the mHC setup~\cite{xie2025mhc}. Notably, our proposed method outperforms mHC even when the latter is configured with 30 Sinkhorn iterations, as further demonstrated in section \ref{sec: appendix_mhc}.

\begin{figure}[t]
    \centering
    \begin{subfigure}[b]{0.45\textwidth}
        \centering
        \includegraphics[width=\textwidth]{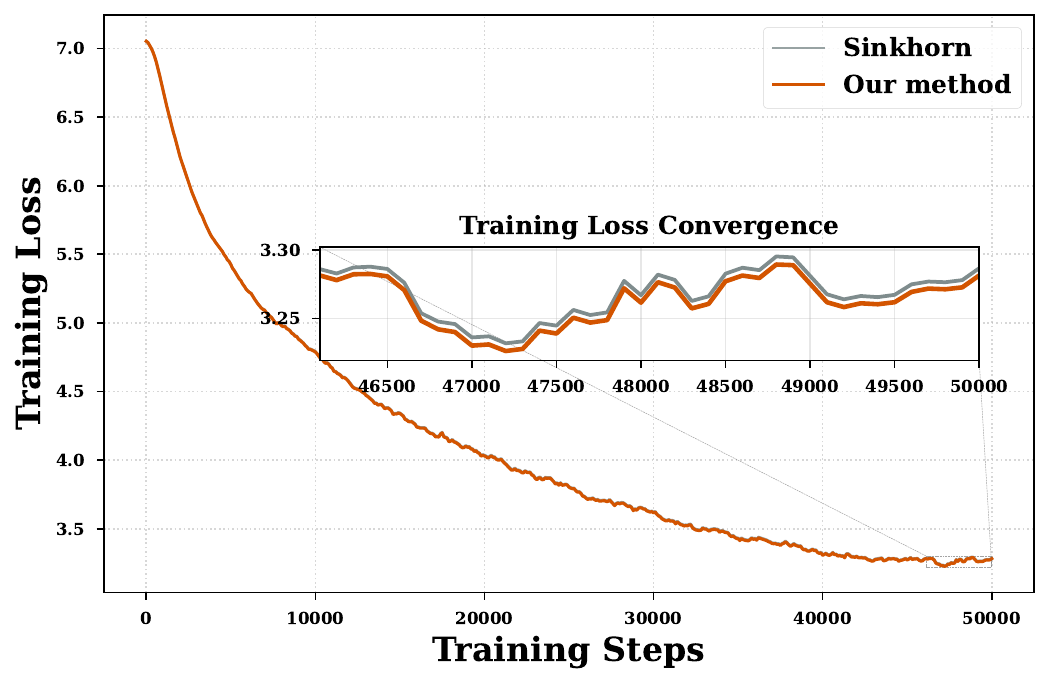}
        \caption{Comparison of training loss trajectories.}
        \label{fig: vit_train_loss}
    \end{subfigure}
    \hfill
    \begin{subfigure}[b]{0.45\textwidth}
        \centering
        \includegraphics[width=\textwidth]{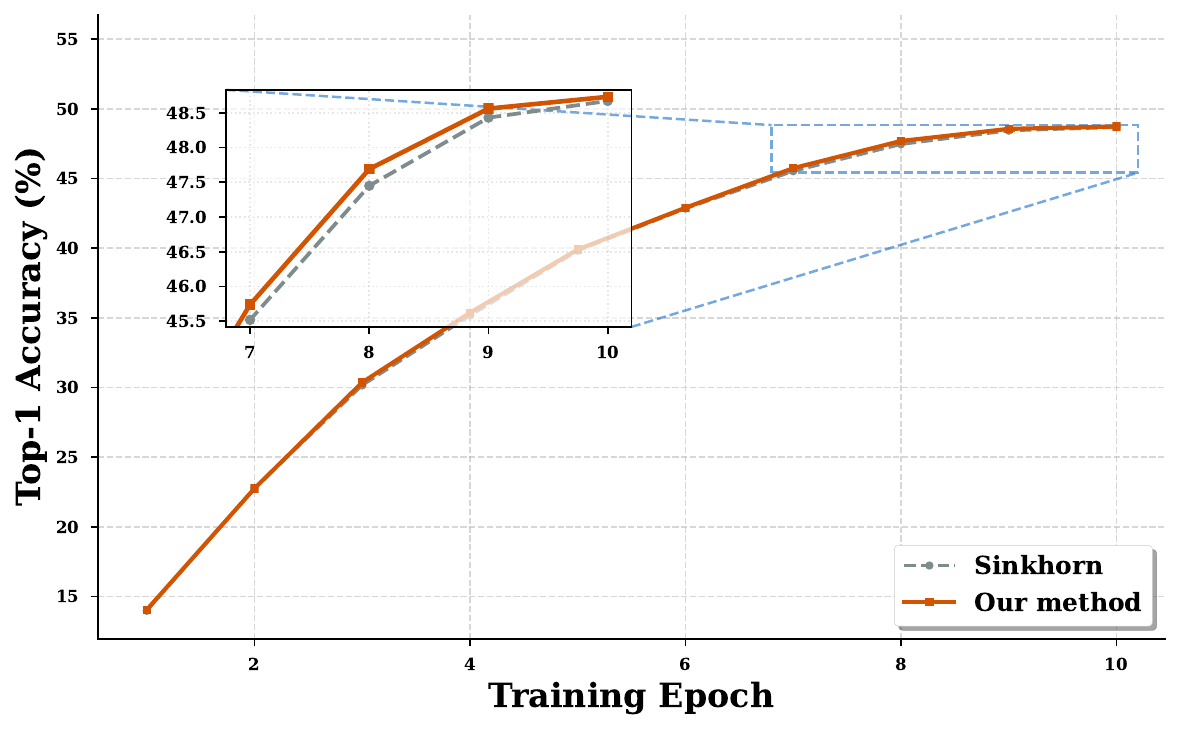}
        \caption{Comparison of Top-1 validation accuracy.}
        \label{fig: vit_valid_acc}
    \end{subfigure}
    \caption{Training loss and validation accuracy comparison on ImageNet.}
    \label{fig: vit_main}
\end{figure}

\section{Conclusion and discussion}

In this paper, we propose the HS-Jacobian based backpropagation method for training DNNs with linear constraints. The forward pass of this layer can be computed by well-developed QP solvers, while the backpropagation relies on the HS-Jacobian of the projection operator. We prove that the HS-Jacobian is a conservative mapping for the projection operator over the polytope. Based on the established conservativity of HS-Jacobian, we further prove the convergence of the HS-Jacobian based Adam algorithm for training linearly constrained DNNs. Extensive numerical experiment results demonstrate the superior performance of the proposed method, compared to the popular unrolling methods. For a potential future research topic, we will focus on the theory and efficient backpropagation implementations for training DNNs with more general constraints. 

\section*{Acknowledgments}

The authors thank Professor Defeng Sun at The Hong Kong Polytechnic University for insightful discussions.


\bibliography{references}{}
\bibliographystyle{siam}

\appendix

\newpage

\section{Preliminaries}
\label{sec: appendix_pre}

\subsection{Semi-algebraic mapping}

\begin{definition}[Semi-algebraic set]
\label{def: semi_algebraic_set}
    A set is said to be semi-algebraic if it is a finite union of sets with the form
    \begin{equation*}
        \bigcap_{i=1}^s ~ \{x\in \mathbb{R}^n ~|~ p_i(x)<0, q_i(x)=0 \},
    \end{equation*}
    where the functions $p_i, q_i: \mathbb{R}^n \to \mathbb{R}$ are real polynomials and $s\ge1$.
\end{definition}

\begin{definition}[Semi-algebraic mapping]
    Let $X\subset \mathbb{R}^n$ and $Y \subset \mathbb{R}^m$ be two semi-algebraic sets, a set-valued mapping $f:X \rightrightarrows Y$ is said to be semi-algebraic if its graph
    \begin{equation*}
        \mathrm{graph}(f) := \{(x,y) ~|~ x\in \mathbb{R}^n, y\in f(x)\}
    \end{equation*}
    is a semi-algebraic set of $\mathbb{R}^n \times \mathbb{R}^m$.
\end{definition}

A useful tool to prove a set is semi-algebraic is the following Tarski-Seidenberg Theorem \cite{bierstone1988semianalytic}, which was summarized from \cite{tarski1952decision, seidenberg1954new}

\begin{lemma}[Theorem 1.5 of \cite{bierstone1988semianalytic}]
\label{lemma: Seidenberg_Tarski}
    Let $X \subseteq \mathbb{R}^{n+1}$ be a semi-algebraic set and $\pi: \mathbb{R}^{n+1} \to \mathbb{R}^n$, the projection on the first $n$ coordinates. Then $\pi(X)$ is a semi-algebraic subset of $\mathbb{R}^n$.
\end{lemma}
The consequence of this theorem can be easily generalized to the following corollary.
\begin{corollary}[Corollary 2.4 of \cite{coste2002introduction}]
\label{corol: projection_semialgebraic}
    If $X \subseteq \mathbb{R}^{n+k}$ is a semi-algebraic set, its image of the projection on the space of the first $n$ coordinates is a semi-algebraic set of $\mathbb{R}^n$.
\end{corollary}

\subsection{Definability}

For the structure $\mathbb{R}_{\mathrm{exp}}$ with language $\mathcal{L}=\{+,\cdot,\mathrm{exp}, <, 0, 1\}$, the set of all $\mathcal{L}$-terms includes constants, variables, and the functions in terms of constants and variables. Given $\mathcal{L}$-terms $t_1,t_2$, the atomic $\mathcal{L}$-formula is of the form $t_1=t_2$ or $t_1<t_2$. The set of all $\mathcal{L}$-formulas is obtained by taking all atomic $\mathcal{L}$-formulas and closing under the Boolean connections $\land, \lor, \lnot$ and quantifiers $\forall, \exists$ \cite{marker1996model, marker2002model}. If a variable is not inside the $\forall, \exists$ quantifiers, then it is a free variable. A set $Z\subset \mathbb{R}^n$ is definable in $\mathbb{R}_{\mathrm{exp}}$ if there is an $\mathcal{L}$-formula $\psi(x_1,\cdots,x_n,y_1,\cdots,y_m)$ with $x_1,\cdots,x_n,y_1,\cdots,y_m$ being free variables and $c_1,\cdots,c_m \in \mathbb{R}$ such that
\begin{equation*}
    Z=\{(z_1,\cdots,z_n) \in \mathbb{R}^n ~|~ \psi(z_1,\cdots,z_n,c_1,\cdots,c_m) \text{ holds}\}.
\end{equation*}
A mapping is definable in $\mathbb{R}_{\mathrm{exp}}$ if its graph is definable in $\mathbb{R}_{\mathrm{exp}}$. One can refer to \cite{coste1999introduction, marker2002model} for more information about definability and $o$-minimal structure. 

\subsection{Conservative mapping}

\begin{definition}[Conservative mappings, \cite{bolte2021conservative}]
\label{def: conservative_mapping}
    Let $f: \mathbb{R}^n \to \mathbb{R}^m$ be a locally Lipschitz function. A nonempty compact set-valued mapping $J_f: \mathbb{R}^n \rightrightarrows \mathbb{R}^{m \times n}$ with closed graph is called a conservative mapping for $f$, if for any absolutely continuous curve $\gamma:[0,1] \to \mathbb{R}^n$, the function $t \mapsto f(\gamma(t))$ satisfies
    \begin{equation*}
        \frac{d}{dt} f(\gamma(t)) = V\dot{\gamma}(t), \quad \forall V\in J_f(\gamma(t)), 
    \end{equation*}
    for almost all $t\in [0,1]$. When $f: \mathbb{R}^n \to \mathbb{R}$ is a scalar-valued function, the corresponding conservative mapping $J_f: \mathbb{R}^n \rightrightarrows \mathbb{R}^n$ is also called a conservative field.
\end{definition}

Based on this definition and Remark 3(e) of \cite{bolte2021conservative}, we can get the following propositions.

\begin{proposition}
\label{prop: subset_mapping_conservative}
    If $J_1,J_2$ are two graph closed set-valued mappings with nonempty compact values, then $J_1 \subset J_2$ and $J_2$ being a conservative mapping implies that $J_1$ is a conservative mapping as well.
\end{proposition}

\begin{proposition}
\label{prop: conv_mapping_conservative}
    If $J: \mathbb{R}^n \rightrightarrows \mathbb{R}^n$ is a conservative mapping, then the mapping
    \begin{equation*}
        \mathrm{conv}(J): \mathbb{R}^n \rightrightarrows \mathbb{R}^n, \quad x\mapsto \mathrm{conv}(J(x))
    \end{equation*}
    is also conservative.
\end{proposition}

Moreover, a nice property of conservative mapping is the product of conservative mappings is still conservative. 
\begin{proposition}[Lemma 5 of \cite{bolte2021conservative}]
    \label{prop: product_conservative_mapping}
    For locally Lipschitz continuous mappings $f_1:\mathbb{R}^p \to \mathbb{R}^m, f_2: \mathbb{R}^m \to \mathbb{R}^l$ and their associated conservative mappings $J_1:\mathbb{R}^p \rightrightarrows \mathbb{R}^{m \times p}, J_2: \mathbb{R}^m \rightrightarrows \mathbb{R}^{l \times m}$, the product mapping $J_2J_1$ is a conservative mapping for $f_2\circ f_1$.
\end{proposition}

Conservative mapping has a very close relation to Clarke subdifferential \cite{clarke1975generalized, clarke1990optimization} as revealed by the following proposition.

\begin{proposition}[Corollary 1 of \cite{bolte2021conservative}]
    \label{prop: clarke_conservative}
    Let $J_f: \mathbb{R}^n \rightrightarrows \mathbb{R}^n$ be a conservative mapping for $f:\mathbb{R}^n \to \mathbb{R}$. Then the Clarke subdifferential $\partial_C f$ is a conservative mapping for $f$ and for all $x \in \mathbb{R}^n$,
    \begin{equation*}
        \partial_C f(x) \subset \mathrm{conv} (J_f(x)).
    \end{equation*}
\end{proposition}

\begin{definition}[Critical point]
    Let $J_f: \mathbb{R}^n \rightrightarrows \mathbb{R}^n$ be a conservative mapping for $f:\mathbb{R}^n \to \mathbb{R}$. Then $x$ is a $J_f$-critical point of $f$ if $0 \in J_f(x)$.
\end{definition}

\begin{proposition}[Proposition 1 of \cite{bolte2021conservative}]
    Let $J_f: \mathbb{R}^n \rightrightarrows \mathbb{R}^n$ be a conservative mapping for $f:\mathbb{R}^n \to \mathbb{R}$ and $x \in \mathbb{R}^n$ be a local minimum or local maximum of $f$. Then $0 \in \mathrm{conv}(J_f(x))$, or equivalently, $x$ is a $\mathrm{conv}(J_f)$-critical point of $f$.
\end{proposition}

\begin{lemma}[Theorem 5 of \cite{bolte2021conservative}]
\label{lemma: nonsmooth_Morse_Sard}
    Let $J_f: \mathbb{R}^n \rightrightarrows \mathbb{R}^n$ be a definable conservative mapping for a definable function $f: \mathbb{R}^n \to \mathbb{R}$. Then the set of $J_f$-critical values $\{f(x) ~|~ 0 \in J_f(x)\}$ is finite.
\end{lemma}

\subsection{Semismoothness}

We present the semismoothness of HS-Jacobian and the equivalence between conservativity and semismoothness in this subsection. The first lemma is an important property of HS-Jacobian.

\begin{lemma}[Lemma 2.1 of \cite{han1997newton}]
\label{prop: HS_Jacobian_semismooth}
    For any $x \in \mathbb{R}^n$, there always exists a neighborhood $U$ of $x$ such that
    \begin{equation*}
        \mathcal{E}(y) \subseteq \mathcal{E}(x), \quad \partial_{HS}\Pi_{\mathcal{P}}(y) \subseteq \partial_{HS}\Pi_{\mathcal{P}}(x), \quad \forall y \in U.
    \end{equation*}
    Moreover, for $\mathcal{E}(y) \subseteq \mathcal{E}(x)$, it holds that
    \begin{equation*}
        \Pi_{\mathcal{P}}(y) = \Pi_{\mathcal{P}}(x) + J (y-x), \quad \forall J \in \partial_{HS}\Pi_{\mathcal{P}}(y).
    \end{equation*}
\end{lemma}

The next lemma builds up the equivalence between semismoothness and conservativity for semi-algebraic mappings.

\begin{lemma}[Theorem 3.5 of \cite{davis2022conservative}]
\label{lemma: equivalence_semismooth_conservative}
    Let $f:\mathbb{R}^n \to \mathbb{R}^m$ be a locally Lipschitz continuous and directionally differentiable mapping and $F:\mathbb{R}^n \times \mathbb{R}^n \rightrightarrows \mathbb{R}^m$ be a set-valued mapping. If $f$ and $F$ are semi-algebraic and $F$ satisfies the following conditions:
    \begin{enumerate}
        \item \label{condition: 1}The image $F(x,u)$ is nonempty and compact for all $x,u \in \mathbb{R}^n$;
        \item \label{condition: 2}The mapping $F$ is positively homogeneous in the second argument, i.e.,
        \begin{equation*}
            F(x,0)=\{0\}, \quad F(x,tu)=tF(x,u),
        \end{equation*}
        for all $x,u \in \mathbb{R}^n$ and $t>0$;
        \item \label{condition: 3}For every $\bar{x} \in \mathbb{R}^n$, there exists $L>0$ such that 
        \begin{equation*}
            \mathrm{dist}(F(x,u_1),F(x,u_2)) \le L\|u_1-u_2\|,
        \end{equation*}
        for all $u_1,u_2\in \mathbb{R}^n$ and all $x$ sufficiently close to $\bar{x}$,
    \end{enumerate}
    then the following statements are equivalent:
    \begin{enumerate}
        \item (\textbf{Semismoothness}) For any $x$ it holds that
        \begin{equation*}
            \limsup\limits_{y\to x} \frac{f(y)-f(x)-F(y,y-x)}{\|y-x\|} = \{0\};
        \end{equation*}
        \item (\textbf{Conservative mapping}) For any absolutely continuous curve $\gamma:[0,1] \to \mathbb{R}^n$, equality holds
        \begin{equation*}
            \left\{\frac{d}{dt}(f\circ\gamma)(t)\right\} = F(\gamma(t), \dot{\gamma}(t)),  
        \end{equation*}
        for almost all $t\in[0,1]$.
    \end{enumerate}
\end{lemma}

\section{Details of Theorem 1 in the main paper}
\label{sec: appendix_thm_1}

We give the proof of Theorem \ref{thm: licq_identical} in this section. 


\begin{proof}
    Let $x\in \mathbb{R}^n$ be a given vector. Denote $y=\Pi_{\mathcal{P}}(x)$. There exist $(\lambda, \mu) \in \mathbb{R}^{m}_+ \times \mathbb{R}^l$ such that $(y,\lambda,\mu)$ satisfies the KKT conditions:
    \begin{equation}
    \label{eq: KKT_system_y}
        \left\{
        \begin{aligned}
            & y - x + A^\top \lambda + B^\top \mu =0, \\
            & A y -a \le 0, \quad By = b, \\
            & \lambda \ge 0, \quad \lambda^\top (Ay - a) = 0.
        \end{aligned}
        \right.
    \end{equation}
    Because LICQ holds at $y$, the corresponding multiplier pair $(\lambda, \mu)$ is unique. For convenience, let $S = \mathrm{supp}(\lambda) = \{i \mid \lambda_i > 0\}$ denote the support of $\lambda$.
    
    The HS-Jacobian $\partial_{HS}\Pi_{\mathcal{P}}(x)$ is defined as a collection of $J_K$ for $K \in \mathcal{E}(x)$, where:
    \begin{equation*}
        J_K = I_n - [A_K^\top ~ B^\top] \left( \begin{bmatrix} A_K \\ B \end{bmatrix} [A_K^\top ~ B^\top] \right)^{-1} \begin{bmatrix} A_K \\ B \end{bmatrix},
    \end{equation*}
    and $\mathcal{E}(x)$ is the family of index sets $K \subseteq \{1,2,\cdots,m\}$ such that $S \subseteq K \subseteq I(x)$ and the rows $\{A_i\}_{i \in K} \cup \{B_j\}_{j=1}^l$ are linearly independent. Under the LICQ condition, any subset $K \subseteq I(x)$ automatically inherits linear independence, resulting in $\mathcal{E}(x) = \{K \mid S \subseteq K \subseteq I(x)\}$.
    
    It has been shown in \cite{han1997newton} that $\partial_B \Pi_{\mathcal{P}}(x) \subseteq \partial_{HS}\Pi_{\mathcal{P}}(x)$. Therefore, to show $\partial_{HS}\Pi_{\mathcal{P}}(x) = \partial_B \Pi_{\mathcal{P}}(x)$, we only need to prove that $\partial_{HS}\Pi_{\mathcal{P}}(x) \subseteq \partial_B \Pi_{\mathcal{P}}(x)$. 
    
    Let $K \in \mathcal{E}(x)$ be any fixed index set, we will show $J_K \in \partial_B \Pi_{\mathcal{P}}(x)$. To achieve this goal, it suffices to construct a sequence $\{x^k\}$ converging to $x$ such that $\Pi_{\mathcal{P}}$ is (Fr\'echet) differentiable at $x^k$ and $\lim_{k \to \infty} \Pi_{\mathcal{P}}'(x^k) = J_K$, where $\Pi_{\mathcal{P}}'(x^k)$ is the Fr\'echet derivative of $\Pi_{\mathcal{P}}(\cdot)$ at $x^k$.
    
    Since the LICQ holds at $y$, there exists $d \neq 0$ such that
    \begin{equation*}
        \left\{
        \begin{aligned}
            A_i d &= 0, \quad \forall i \in K \nonumber \\
            B d &= 0, \nonumber \\
            A_i d & < 0, \quad \forall i \in I(x) \setminus K.
        \end{aligned}
        \right.
    \end{equation*}
    Let $\{t_k\}$ be a scalar sequence such that $t_k \downarrow 0$ and define $y^k = y + t_k d$. Then for sufficiently large $k$, we have
    \begin{equation}
    \label{eq: perturbed_active_set}
        \left\{
        \begin{aligned}
            & A_iy^k = a_i, & & \forall i \in K, \\
            & A_iy^k < a_i & & \forall i \notin K.
        \end{aligned}
        \right.
    \end{equation}
    
    Next, we define $(\lambda^k, \mu^k)$ as
    \begin{equation*}
        \lambda^k_i = 
        \begin{cases} 
          \lambda_i + t_k, & \text{if } i \in K, \\
          0, & \text{if } i \notin K,
       \end{cases} \quad \mu^k = \mu,
    \end{equation*}
    and $x^k = y^k + A^\top \lambda^k + B^\top \mu^k$. Therefore, for sufficiently large $k$, we have
    \begin{equation}
    \label{eq: perturbed_support_set}
        \left\{
        \begin{aligned}
            & \lambda_i^k = \lambda_i + t_k >0, & & \forall i \in S, \\
            & \lambda_i^k = t_k > 0, & & \forall i \in K \setminus S, \\
            & \lambda_i^k = 0, & & \forall i \notin K.
        \end{aligned}
        \right.
    \end{equation}

    Thus, based on \eqref{eq: perturbed_active_set} and \eqref{eq: perturbed_support_set}, we can observe that $(x^k,y^k,\lambda^k, \mu^k)$ satisfies the KKT system \eqref{eq: KKT_system_y}, which implies that $\Pi_{\mathcal{P}}(x^k)=y^k$. Since $t_k \downarrow 0$, we have $y^k \to y$ and $\lambda^k \to \lambda$. Consequently, $\lim_{k \to \infty} x^k = x$. Moreover, for all sufficiently large $k$, we have that $I(x^k) = K$ and $\mathrm{supp}(\lambda^k) = K$, which implies that the strictly complementary condition holds. Therefore, the projection operator $\Pi_{\mathcal{P}}$ is differentiable at $x^k$ \cite[4.1.2 Corollary]{facchinei2007finite}, and $\Pi_{\mathcal{P}}'(x^k)=J_K$. Taking the limit yields $\lim_{k \to \infty} \Pi_{\mathcal{P}}'(x^k) = J_K$. Thus, we have shown that $J_K \in \partial_B \Pi_{\mathcal{P}}(x)$. Because $K \in \mathcal{E}(x)$ was chosen arbitrarily, we have $\partial_{HS}\Pi_{\mathcal{P}}(x) \subseteq \partial_B \Pi_{\mathcal{P}}(x)$. Therefore, we establish that $\partial_{HS}\Pi_{\mathcal{P}}(x) = \partial_B \Pi_{\mathcal{P}}(x)$ under the LICQ.
\end{proof}

\section{Details of Theorem 2 in the main paper}
\label{sec: appendix_thm_2}

We give the proof of Theorem \ref{thm: semi_algebraic_Jacobian} in this section. 


\subsection{Supporting lemma}

\begin{lemma}
\label{lemma: semi_algebraic_projection}
    The projection operator $\Pi_{\mathcal{P}}: \mathbb{R}^n \to \mathbb{R}^n$ is a semi-algebraic mapping.
\end{lemma}

\begin{proof}
    Consider the set in $\mathbb{R}^n \times \mathbb{R}^n \times \mathbb{R}^m \times \mathbb{R}^l$
    \begin{equation*}
        \Gamma:=\left\{ (x,y,\lambda,\mu) ~ \middle|
        \begin{aligned}
            & ~ y-x+A^\top \lambda + B^\top \mu = 0,\\
            & ~ Ay-a \le 0, \quad By=b,\\
            & ~ \lambda \ge 0, \quad \lambda^\top (Ay-a)=0,
        \end{aligned}
        \right\}
    \end{equation*}
    Since this set is defined by finite polynomial equations and inequalities, it is a semi-algebraic set (see Definition \ref{def: semi_algebraic_set}). By Corollary \ref{corol: projection_semialgebraic}, the projection of $\Gamma$ onto the space of the first two variables 
    \begin{equation*}
        \Gamma_{xy}=\{(x,y) ~|~ \exists ~(\lambda,\mu) \in \mathbb{R}^m \times \mathbb{R}^l, \text{s.t. } (x,y,\lambda,\mu) \in \Gamma \}
    \end{equation*}
    is also a semi-algebraic set. By the KKT condition, it holds that
    \begin{equation*}
        \mathrm{graph}(\Pi_{\mathcal{P}}) = \{(x,y) ~|~ x\in \mathbb{R}^n, y = \Pi_{\mathcal{P}}(x)\}
    \end{equation*}
    coincides with $\Gamma_{xy}$. This completes the proof.
\end{proof}

\subsection{The proof of Theorem \ref{thm: semi_algebraic_Jacobian}}

\begin{proof}
    Given an index set $K\subseteq \{1,2,\cdots,m\}$, let $H_K=[A_K^\top ~ B^\top]$, where $A_K$ is composed of rows of $A$ indexed by $K$. Define
    \begin{equation*}
        J_K := I_n-H_K(H_K^\top H_K)^{-1} H_K^\top, \quad S_K := \{x ~|~ K\in \mathcal{E}(x)\}.
    \end{equation*}
    The graph of $\partial_{HS}\Pi_{\mathcal{P}}$ is 
    \begin{equation*}
        \begin{aligned}
            \mathrm{graph}(\partial_{HS}\Pi_{\mathcal{P}}) &= \{(x,J) ~|~ x\in \mathbb{R}^n, J \in \partial_{HS}\Pi_{\mathcal{P}}(x)\} \\
            &= \{(x,J) ~|~ x\in \mathbb{R}^n,J=J_K, K \in \mathcal{E}(x)\}.
        \end{aligned}
    \end{equation*}  
    We claim that
    \begin{equation}
    \label{eq: HS_Jacobian_graph}
    \mathrm{graph}(\partial_{HS}\Pi_{\mathcal{P}}) = \bigcup_{K\subseteq \{1,2,\cdots,m\}} \left( S_K \times \{J_K\} \right).
    \end{equation}
    For any $(x,J)\in \mathrm{graph}(\partial_{HS}\Pi_{\mathcal{P}})$, there exists an index set $K\subseteq \{1,2,\cdots,m\}$ and $K\in \mathcal{E}(x)$ such that $J=J_K$. This is equivalent to $x\in S_K$ and $J=J_K$. In contrast, for any $(x,J)$ from the right hand side of \eqref{eq: HS_Jacobian_graph}, there exists an index set $K \subseteq \{1,2,\cdots,m\}$ such that $x\in S_K$ and $J=J_K$. Therefore, the equation \eqref{eq: HS_Jacobian_graph} holds. 
    
    Based on the fact that $\mathrm{graph}(\partial_{HS}\Pi_{\mathcal{P}})$ is of the form of finite unions and $J_K$ is a constant matrix for fixed $K$, it suffices to prove $S_K$ is a semi-algebraic set. Given an index set $K\subseteq \{1,2,\cdots,m\}$, if $H_K$ is not of full column rank, then $S_K=\varnothing$. So we only consider the case where $H_K$ is of full column rank. In this case, $x \in S_K$ if and only if
    \begin{equation*}
        \exists ~ (\lambda,\mu) \in \mathbb{R}^m_+ \times \mathbb{R}^l, \quad \mathrm{s.t.} ~ (x,\lambda,\mu) \text{ satisfies KKT system and } \mathrm{supp}(\lambda) \subseteq K \subseteq I(x).
    \end{equation*}
    This implies that $S_K$ is a projection set of 
    \begin{equation*}
        \left\{ (x,\lambda,\mu) ~ \middle| ~
        \begin{aligned}
            & \Pi_{\mathcal{P}}(x)-x + A^\top \lambda + B^\top \mu = 0 \\
            & A_K \Pi_{\mathcal{P}}(x) = a_K, \quad A_{\bar{K}} \Pi_{\mathcal{P}}(x) \le a_{\bar{K}} \\
            & B \Pi_{\mathcal{P}}(x) = b \\
            & \lambda_K \ge 0, \quad \lambda_{\bar{K}} = 0
        \end{aligned}
        \right\},
    \end{equation*}
    where $\bar{K} := \{1,2,\cdots,m\} \setminus K$ and $A_K, A_{\bar{K}}, a_K, a_{\bar{K}}, \lambda_K, \lambda_{\bar{K}}$ represent the submatrices and subvectors indexed by the subscripts, respectively. Since $\Pi_{\mathcal{P}}$ is a semi-algebraic mapping, we can observe that this set is semi-algebraic. Consequently, by a similar argument in the proof of Lemma \ref{lemma: semi_algebraic_projection}, its projection set $S_K$ is also a semi-algebraic set.

    Therefore, for a given index set $K \subseteq \{1,2,\cdots,m\}$, $S_K$ is a semi-algebraic set. Based on this fact we conclude that the HS-Jacobian $\partial_{HS} \Pi_{\mathcal{P}}$ is a semi-algebraic mapping.

    Now, we focus on proving the conservativity of HS-Jacobian. Obviously, by the definition of HS-Jacobian, it is a nonempty compact set-valued mapping. We first prove that it is also graph closed. This is equivalent to prove for any $x \in \mathbb{R}^n$ and any converging sequence $\{x_k\}$ satisfying $\lim_{k \to \infty} x_k =x$, if $J_k \in \partial_{HS}\Pi_{\mathcal{P}}(x_k)$ and $\lim_{k \to \infty} J_k = J$, then $J \in \partial_{HS}\Pi_{\mathcal{P}}(x)$. Based on Proposition \ref{prop: HS_Jacobian_semismooth}, for any $x\in \mathbb{R}^n$, there exists a neighborhood $U$ of $x$ such that $\partial_{HS}\Pi_{\mathcal{P}}(y) \subset \partial_{HS}\Pi_{\mathcal{P}}(x)$ for any $y\in U$. Thus, we have that for sufficiently large $k$, $x_k \in U$. For any $J_k \in \partial_{HS}\Pi_{\mathcal{P}}(x_k)$ and $\lim_{k \to \infty} J_k = J$, we have $J_k \in \partial_{HS}\Pi_{\mathcal{P}}(x_k) \subset \partial_{HS}\Pi_{\mathcal{P}}(x)$. Since $\partial_{HS}\Pi_{\mathcal{P}}(x)$ has finite elements, it is obvious that $J \in \partial_{HS}\Pi_{\mathcal{P}}(x)$. This indicates HS-Jacobian is graph closed.

    Define a set-valued mapping based on HS-Jacobian 
    \begin{equation*}
        F: \mathbb{R}^n \times \mathbb{R}^n \rightrightarrows \mathbb{R}^n, (x,u) \mapsto \{Ju ~|~ \forall J \in \partial_{HS}\Pi_{\mathcal{P}}(x)\}.
    \end{equation*}
    By Proposition \ref{prop: HS_Jacobian_semismooth}, for any $x\in \mathbb{R}^n$, there exists a neighborhood $U$ such that for any $y \in U$, $\mathcal{E}(y) \subseteq \mathcal{E}(x)$ and 
    \begin{equation*}
        \Pi_{\mathcal{P}}(y) = \Pi_{\mathcal{P}}(x) + J(y-x), \quad \forall J\in \partial_{HS}\Pi_{\mathcal{P}}(y).
    \end{equation*}
    This means 
    \begin{equation*}
        F(y,y-x) = \{\Pi_{\mathcal{P}}(y) - \Pi_{\mathcal{P}}(x)\},
    \end{equation*}
    which implies that the semismoothness in Lemma \ref{lemma: equivalence_semismooth_conservative} holds. Also, it is known that $\Pi_{\mathcal{P}}$ is Lipschitz continuous, directionally differentiable \cite{haraux1977differentiate} and semi-algebraic (see Lemma \ref{lemma: semi_algebraic_projection}). 

    Moreover, by a similar argument as in the previous proof, the graph of $F$ can be represented as 
    \begin{equation*}
        \begin{aligned}
            \mathrm{graph}(F) 
            &= \{(x,u,y) ~|~ \forall (x,u) \in \mathbb{R}^n \times \mathbb{R}^n , y \in F(x,u)\} \\
            &= \{(x,u,y) ~|~ \forall (x,u) \in \mathbb{R}^n \times \mathbb{R}^n , y=Ju, \forall J \in \partial_{HS}\Pi_{\mathcal{P}}(x)\} \\
            &= \{(x,u,y) ~|~ \forall (x,u) \in \mathbb{R}^n \times \mathbb{R}^n , y=J_Ku, \forall K\in \mathcal{E}(x)\} \\
            &= \bigcup_{K\subseteq\{1,2,\cdots,m\}} (S_K \times \{(u,y) ~|~ y=J_Ku\}).
        \end{aligned}
    \end{equation*}
    Therefore, we know that $F$ is also semi-algebraic as $S_K$ is a semi-algebraic set.

    To prove HS-Jacobian is a conservative mapping, it suffices to prove that $F$ satisfies the three conditions in Lemma \ref{lemma: equivalence_semismooth_conservative}. By the definition of $F$, the condition \ref{condition: 2} is automatically fulfilled. Moreover, since $|\partial_{HS}\Pi_{\mathcal{P}}(x)|$ is finite for any $x\in \mathbb{R}^n$, it is easy to verify that the condition \ref{condition: 1} and \ref{condition: 3} also hold. Thus, we have that for any absolutely continuous curve $\gamma : [0,1] \to \mathbb{R}^n$, 
    \begin{equation*}
        \frac{d}{dt} (\Pi_{\mathcal{P}} \circ \gamma)(t) = J\dot{\gamma}(t), \quad \forall J\in \partial_{HS}\Pi_{\mathcal{P}}(\gamma(t)),
    \end{equation*}
    for almost all $t\in[0,1]$. This indicates HS-Jacobian is a conservative mapping for projection operator $\Pi_{\mathcal{P}}$. 
\end{proof}

\section{Details of Theorem 3 in the main paper}
\label{sec: appendix_thm_3}

We give the proof of Theorem \ref{thm: adam_convergence} in this section.

\subsection{Supporting lemmas}

\begin{lemma}
\label{lemma: each_AD_mapping_measurable}
    For a given $(x^{(t)}, y^{(t)}) \in \Omega$, the AD field $\mathcal{D}_{\varphi_t}: \mathbb{R}^p \rightrightarrows \mathbb{R}^p$ for $\varphi_t(\theta)$ is a measurable mapping\footnote{For the definition of measurable set-valued mappings, one can refer to Section 17 and 18 of \cite{aliprantis2006infinite} for more details.}.
\end{lemma}
\begin{proof}
    Noticing that a conservative mapping is a graph-closed mapping by definition, by \cite[Theorem 18.20]{aliprantis2006infinite}, a graph-closed mapping between two $\sigma$-compact Hausdorff spaces is measurable. This completes the proof.
\end{proof}

\begin{lemma}
\label{lemma: conservative_mapping_jointly_measurable}
    The set-valued mapping $\mathcal{D}: \mathbb{R}^p \times \Omega \rightrightarrows \mathbb{R}^p, (\theta,x,y) \mapsto \mathcal{D}_{\Phi(\cdot;x,y)}(\theta)$ is a measurable mapping.
\end{lemma}
\begin{proof}
    For any closed set $U \subset \mathbb{R}^p$, its lower inverse of $\mathcal{D}$ \cite{aliprantis2006infinite} can be expressed as
    \begin{equation*}
        \begin{aligned}
            \mathcal{D}^l (U) 
            &= \{(\theta, x,y) ~|~ \mathcal{D}(\theta,x,y) \cap U \ne \varnothing \} \\
            &= \bigcup_{t=1}^N \left( \{\theta ~|~ \mathcal{D}(\theta,x^{(t)}, y^{(t)}) \cap U \ne \varnothing\} \times \{(x^{(t)}, y^{(t)})\} \right) \\
            &= \bigcup_{t=1}^N \left( \{\theta ~|~ \mathcal{D}_{\varphi_t}(\theta) \cap U \ne \varnothing\} \times \{(x^{(t)}, y^{(t)})\} \right) \\
            &= \bigcup_{t=1}^N \left( \mathcal{D}_{\varphi_t}^l(U) \times \{(x^{(t)}, y^{(t)})\} \right).
        \end{aligned}
    \end{equation*}
    Since $\mathcal{D}_{\varphi_t}$ is measurable by Lemma \ref{lemma: each_AD_mapping_measurable} and $\{(x^{(t)}, y^{(t)})\} \in \mathcal{F}$, we know that $\mathcal{D}^l(U)$ is also measurable. This demonstrates $\mathcal{D}$ is a measurable mapping.
\end{proof}

\subsection{The proof of Theorem \ref{thm: adam_convergence}}

\begin{proof}
    To prove this theorem, it suffices to show that under all assumptions in Theorem \ref{thm: adam_convergence}, all the conditions in \cite[Corollary 1]{xiao2024adam} are satisfied. 

    We first prove the measurability of the set-valued mapping $\mathcal{D}: \mathbb{R}^p \times \Omega \rightrightarrows \mathbb{R}^p$ and the mapping 
    \begin{equation*}
        \Psi: \mathbb{R}^p \times \Omega \to \mathbb{R}, \quad (\theta,x,y) \mapsto \Phi(\theta;x,y).
    \end{equation*}
    By Lemma \ref{lemma: conservative_mapping_jointly_measurable}, the measurability of $\mathcal{D}$ can be derived. Based on a similar argument in the proof of Lemma \ref{lemma: conservative_mapping_jointly_measurable}, the measurability of $\Psi$ can be obtained by the measurability of the mapping $\varphi_t(\cdot) = \Phi(\cdot; x^{(t)}, y^{(t)})$ for each fixed $(x^{(t)}, y^{(t)}) \in \Omega$. By Assumption \ref{assum: block_definable}, $\varphi_t$ is a continuous mapping, resulting in its measurability. 

    As discussed in the main paper, for each fixed $(x^{(t)}, y^{(t)}) \in \Omega$, the AD field $\mathcal{D}_{\varphi_t}$ is a definable conservative mapping for the definable function $\varphi_t$. 
    
    Furthermore, based on Definition \ref{def: conservative_mapping} and Proposition \ref{prop: conv_mapping_conservative}, it is easy to check that $\mathcal{D}_{\varphi}$ \eqref{eq: conv_of_conservative_for_loss} is also a conservative mapping for the loss function $\varphi$ \eqref{eq: empirical_objective_function}. We then discuss the definability of the loss function $\varphi$ and its conservative field $\mathcal{D}_{\varphi}$. Obviously, the definability of $\varphi$ can be obtained by the definability of $\varphi_t$. Based on the definability of AD fields $\mathcal{D}_{\varphi_t}(\theta)$, the definability of $\frac{1}{N}\sum_{t=1}^N \mathcal{D}_{\varphi_t}(\theta)$ can be derived. Moreover, by Carath\'{e}odory's theorem \cite[Theorem 17.1]{rockafellar1970convex}, the point from a convex hull in $\mathbb{R}^p$ can be expressed as a convex combination of at most $p+1$ points from the original set. Based on this fact and the definition of definable set, we can obtain that the conservative mapping $\mathcal{D}_{\varphi}$ is also definable. 
    
    Thus, $\mathcal{D}_{\varphi}$ is a definable conservative mapping for the definable function $\varphi$. By Lemma \ref{lemma: nonsmooth_Morse_Sard}, we have that $\{\varphi(\theta) ~|~ 0 \in \mathcal{D}_{\varphi}(\theta)\}$ is finite, which indicates that $\{\varphi(\theta) ~|~ 0 \in \mathcal{D}_{\varphi}(\theta)\}$ has empty interior in $\mathbb{R}$. 

    Combined with Assumption \ref{assum: conservative_mapping_setting} and \ref{assum: adam_algorithm_setting}, all the other conditions in Corollary 1 of \cite{xiao2024adam} are satisfied. Thus, the convergence property of Algorithm 3 is guaranteed.
\end{proof}

\section{Portfolio allocation}
\label{sec: appendix_portfolio}

In the portfolio allocation experiment, we have used the following settings.

\noindent \textbf{Model.} The StemGNN \cite{Cao2020spectral} is used here to predict both the future asset returns and current portfolio weight. For a predicted portfolio weight $w=(w_1,\cdots,w_n)^\top \in \mathbb{R}^n$, the linear constraints for $w$ are
\begin{equation*}
    \sum_{i=1}^n w_i = 1, \quad \sum_{i \in \mathcal{C}}w_i \ge \delta, \quad w\ge 0.
\end{equation*} 
Both LinSATNet and the proposed projection layer are added on $w$ to ensure its feasibility. The training loss is a weighted sum of maximizing the Sharpe ratio \cite{sharpe1966mutual, sharpe1994sharpe} and minimizing the error between the predicted returns and the true returns. The loss weight is set to 1 in the experiment. With out loss of generality, the expert preference set is set to $\mathcal{C}=\{1,2,\cdots,5\}$ and the minimum total weight is set to $\delta=0.5$ in this experiment. 

\noindent \textbf{Forward pass of projection layer.} Gurobi \cite{gurobi} is used here to solve the QP problem. Both the feasibility tolerance and optimality tolerance are set to $1\times 10^{-8}$.

\noindent \textbf{Dataset.} In this experiment, the training dataset consists of daily prices of 493 assets from S\&P 500 index from 2018-01-01 to 2020-12-30. The test dataset covers the data from 2021-03-01 to 2021-12-30. The window size for constructing samples is set to 120. During training, each input sample comprises the prices of all assets over the preceding 120 days and the outputs are the predicted returns of the next 120 days and the current portfolio weights. 

\noindent \textbf{Evaluation.} Following \cite{wang2023linsatnet}, the model performance is evaluated by the Sharpe ratio \cite{sharpe1966mutual, sharpe1994sharpe}. The corresponding returns and risk are computed based on the predicted portfolio allocation weight and the ground truth asset data in the evaluation period. The risk-free return is set to 0.03 in the experiment.

\noindent \textbf{Feasibility computation.} 
In the experiment, the feasibility violation $V^{\mathrm{all}}$ for each sample is computed by
\begin{equation*}
    V^{\mathrm{all}} = \max\{V^{\mathrm{ineq}}_1, V^{\mathrm{ineq}}_2, V^{\mathrm{eq}}\},
\end{equation*}
where
\begin{equation*}
        V^{\mathrm{ineq}}_1 = \left\| \min \left\{\sum_{i\in \mathcal{C}}w_i-\delta, 0\right\} \right\|_2^2, \quad V^{\mathrm{ineq}}_2 = \| \min \{ w,0 \} \|_2^2, \quad V^{\mathrm{eq}} = \left\| \sum_{i=1}^nw_i-1 \right\|_2^2.
\end{equation*}
The feasibility violation for each epoch is the average of feasibility violation of all samples in this epoch. 

\noindent \textbf{Other hyperparameters.} Following \cite{wang2023linsatnet}, we set other hyperparameters as shown in Table~\ref{tab: portfolio_train_param}: 

\begin{table}[ht]
    \centering
    \caption{Training hyperparameters for portfolio allocation.}
    \vspace{1em} 
    
    \definecolor{tablegray}{gray}{0.85}
    
    \renewcommand{\arraystretch}{1.2}
    
    \begin{tabular}{>{\columncolor{tablegray}\bfseries}l l}
        \toprule
        \rowcolor{tablegray} \text{Hyperparameter} & \textbf{Value} \\
        \midrule
        Batch size         & 32  \\
        Epochs             & 50 \\
        Learning rate      & 1e-5 \\
        Learning rate decay coefficient  & 0.5 \\
        Learning rate decay frequency    & 5 epochs \\
        Temperature hyperparameter in LinSATNet & $\tau=0.01$ \\
        Maximum iterations in LinSATNet & 100 \\
        \bottomrule
    \end{tabular}
    \label{tab: portfolio_train_param}
\end{table}

\section{Graph matching}
\label{sec: appendix_graph}

We details the settings of graph matching experiment in this subsection. 

\noindent \textbf{Model.} The neural graph matching v2 (NGM-v2) \cite{wang2021neural} model is selected here to extract the features of images and predict the matching score matrix $X \in \mathbb{R}^{d_1 \times d_2}$. Specifically, this model consists of a CNN extractor and an NGM solver. In the CNN extractor, the VGG16 net \cite{simonyan2015very} is adopted to extract node features and global features from the input images. The SplineCNN \cite{fey2018splinecnn} is then utilized to refine node features, which are further transformed to edge features. Combined with global features, the refined node features and edge features are then transformed to node similarity matrix and edge similarity matrix, based on which the association graph \cite{leordeanu2005spectral, cho2010reweighted} corresponding to the graph matching problem is obtained. The NGM solver is a network, whose input is the affinity matrix of the association graph and the output is the matching score matrix. 

To ensure that the matching score matrix is a doubly stochastic matrix, original NGM-v2 model exploit a Sinkhorn layer. Different from the original NGMv2 model, in partial graph matching tasks, the matching score is imposed the following linear constraints
\begin{equation*}
    \sum_{j=1}^{d_2} X_{ij} \le 1, \forall i\in \{1,\cdots,d_1\}, \quad \sum_{i=1}^{d_1} X_{ij} \le 1, \forall j\in \{1,\cdots, d_2\}, \quad \sum_{i=1}^{d_1} \sum_{j=1}^{d_2} X_{ij} \le \alpha, \quad X\ge 0.
\end{equation*}
To ensure the matching score $X$ satisfy those constraints, in this experiment, we replace the original Sinkhorn layer in \cite{wang2021neural} with LinSATNet \cite{wang2023linsatnet} or the proposed projection layer. With slight abuse of notation, we still use $X$ to denote the outputs of LinSATNet or the projection layer. Following \cite{wang2021neural}, we use the binary cross entropy as the training loss and conduct an end-to-end training from scratch. However, in the early training stage, the training is not stable due to the fact that, the initial model may generate the predictions opposite to the ground truth labels, thereby causing the gradient explosion \cite{tan2023equalization}. To alleviate this problem and stabilize the training process, we add an additional clamp operator on the predicted matching score matrix $X$ by
\begin{equation*}
    \tilde{X}_{ij} = \left\{ \begin{aligned}
        & \varepsilon, & ~ & \mathrm{if} ~ 0\le X_{ij} \le \varepsilon, \\
        & 1-\varepsilon, & ~ & \mathrm{if} ~ 1-\varepsilon \le X_{ij} \le 1, \\
        & X_{ij}, & ~ & \mathrm{otherwise}.
    \end{aligned} \right.
\end{equation*}
The training loss is based on the binary cross entropy between clamped matching score matrix $\tilde{X}$ and the ground truth permutation matrix.

\noindent \textbf{Forward pass of projection layer.} Gurobi \cite{gurobi} is used here to solve the QP problem. Both the feasibility tolerance and optimality tolerance are set to $1\times 10^{-8}$.

\noindent \textbf{Dataset.} We adopt the Pascal VOC dataset \cite{everingham2010pascal} with Berkeley annotations \cite{bourdev2009poselets} of keypoints in the graph matching experiment. The original Pascal VOC dataset contains images with bounding boxes surrounding objects. The keypoints of those objects are added in the Berkeley annotations. It contains 20 categories of objects with annotated keypoint locations. All objects with annotated keypoints are cropped around its bound box and resized to $256\times 256$. The number of keypoints for each object ranges from 6 to 23. It is worth mentioning that there are some keypoints of the objects that are out of the original bounding boxes. In the experiment, we follow the unfiltered settings used in \cite{rolinek2020deep, wang2023linsatnet}. That means, during training, for a given pair of images, the resized images may contain outlier keypoints. So this setting is suitable for partial graph matching tasks.

\noindent \textbf{Evaluation.} Following \cite{rolinek2020deep, wang2023linsatnet}, the model performance is evaluated based on F1-score. When evaluating, each predicted matching score matrix is first transformed into permutation matrix by Hungarian algorithm \cite{kuhn1955hungarian}. The the F1-score is computed based on the transformed permutation matrix and the ground truth permutation matrix.

\noindent \textbf{Feasibility computation.} We first reformulate the linear constraints \eqref{eq: cons_graph} into a vector form
\begin{equation*}
    A\mathrm{vec}(X) \le a, \quad \mathrm{vec}(X) \ge 0,
\end{equation*}
where $\mathrm{vec}(\cdot)$ is the row-major vectorization operator,
\begin{equation*}
    A=\begin{pmatrix}
        I_{d_1} \otimes \textbf{1}_{d_2}^\top \\
        \textbf{1}_{d_1}^\top \otimes I_{d_2} \\
        \textbf{1}_{d_1d_2}^\top
    \end{pmatrix} \in \mathbb{R}^{(d_1+d_2+1) \times (d_1d_2)}, \quad a=\begin{pmatrix}
        \textbf{1}_{d_1} \\
        \textbf{1}_{d_2} \\
        \alpha
    \end{pmatrix} \in \mathbb{R}^{d_1+d_2+1},
\end{equation*}
and $\otimes$ is the Kronecker product. The feasibility violation $V^{\mathrm{all}}$ for one sample is computed by
\begin{equation*}
    V^{\mathrm{all}} = \max\{V_1^{\mathrm{ineq}}, V_2^{\mathrm{ineq}}\}
\end{equation*}
where 
\begin{equation*}
    V_1^{\mathrm{ineq}} = \|\max\{A\mathrm{vec}(X)-a, 0\} \|_2^2, \quad V_2^{\mathrm{ineq}} = \|\min\{\mathrm{vec}(X), 0\}\|_2^2.
\end{equation*}
The feasibility violation for each epoch is the average of feasibility violation of all samples in this epoch.

\noindent \textbf{Other hyperparameters.} Following \cite{wang2021neural}, we set other hyperparameters as shown in Table~\ref{tab: graph_train_param}: 

\begin{table}[ht]
    \centering
    \caption{Training hyperparameters for graph matching.}
    \vspace{1em} 
    
    \definecolor{tablegray}{gray}{0.85}
    
    \renewcommand{\arraystretch}{1.2}
    
    \begin{tabular}{>{\columncolor{tablegray}\bfseries}l l}
        \toprule
        \rowcolor{tablegray} \text{Hyperparameter} & \textbf{Value} \\
        \midrule
        Batch size         & 4  \\
        Epochs             & 20 \\
        CNN extractor learning rate   & 2e-5 \\
        NGM solver learning rate      & 2e-3 \\
        Clamp coefficient  & 1e-3 \\
        Maximum matching number $\alpha$ & Ground truth matching number \\
        Temperature hyperparameter in LinSATNet & $\tau=0.2$ \\
        Maximum iterations in LinSATNet & 80 \\
        \bottomrule
    \end{tabular}
    \label{tab: graph_train_param}
\end{table}

\noindent \textbf{More visualization results of matched graphs.} More matching results are visualized in Fig.~\ref{fig: graph_visual_SI_all}.

\begin{figure}[htb]
    \centering

    \begin{subfigure}{1\textwidth}
        \centering
        \includegraphics[width=1\textwidth]{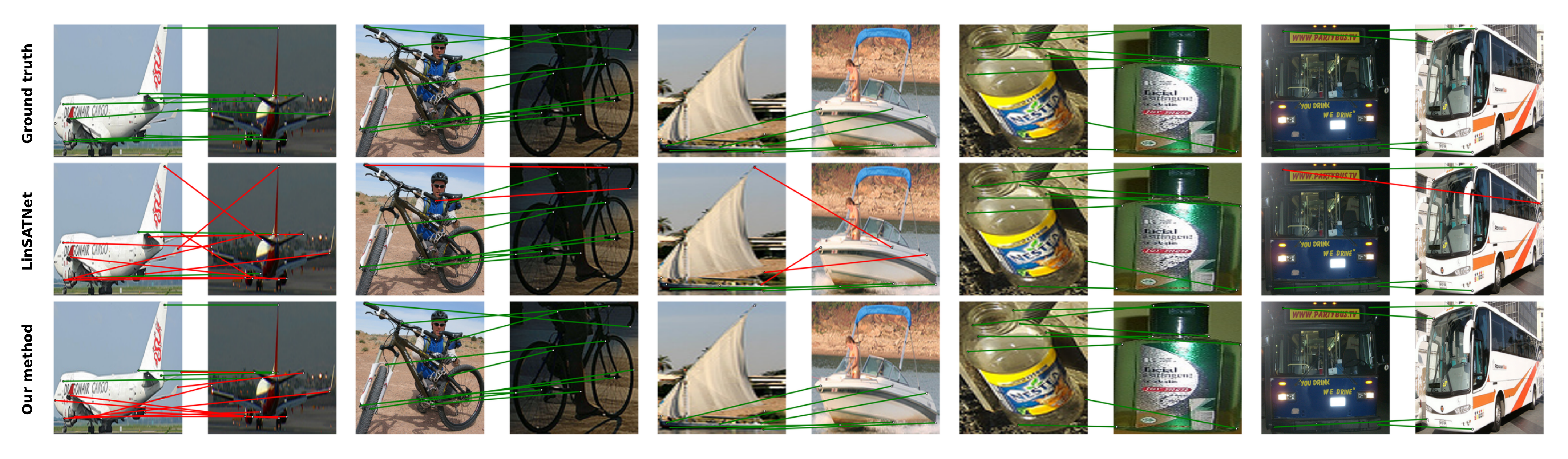}
        \caption{The corresponding categories of matched images are aeroplane, bicycle, boat, bottle, and bus.}
        \label{fig: graph_visual_SI_1}
    \end{subfigure}

    \begin{subfigure}{1\textwidth}
        \centering
        \includegraphics[width=1\textwidth]{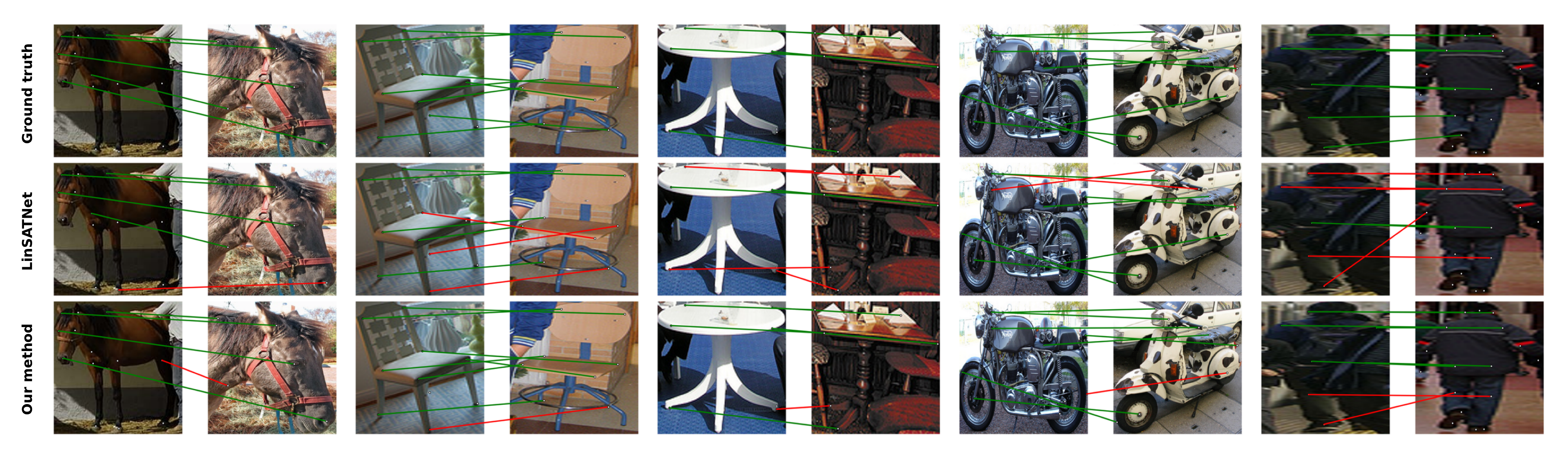}
        \caption{The corresponding categories of matched images are horse, chair, diningtable, motorbike, person.}
        \label{fig: graph_visual_SI_2}
    \end{subfigure}
        
    \begin{subfigure}{1\textwidth}
        \centering
        \includegraphics[width=1\textwidth]{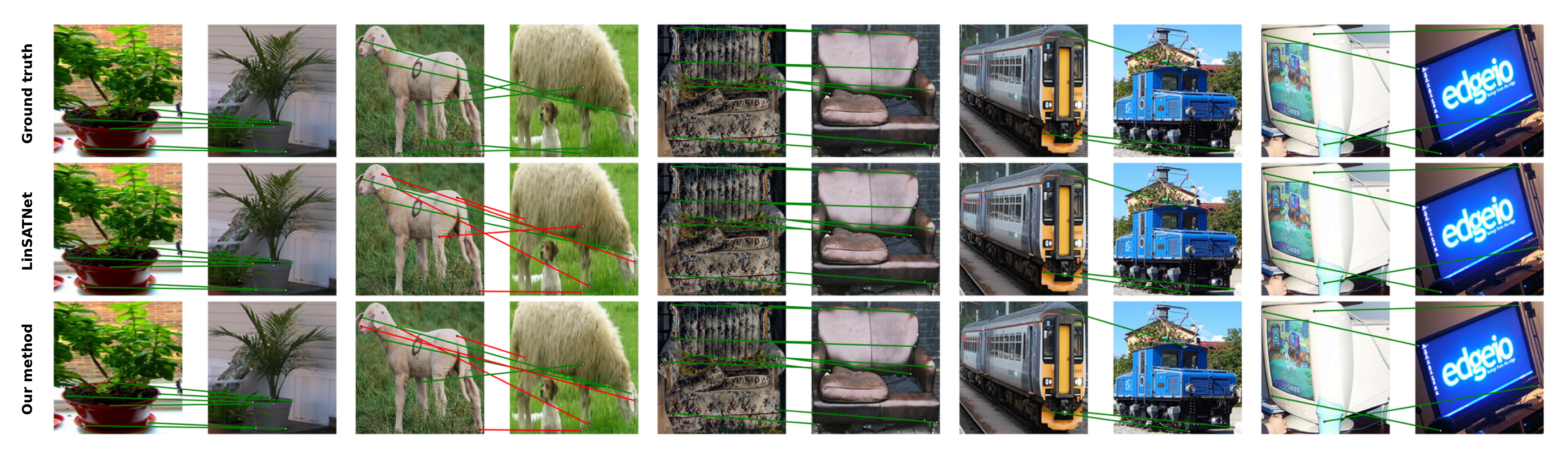}
        \caption{The corresponding categories of matched images are pottedplant, sheep, sofa, train, tvmonitor.}
        \label{fig: graph_visual_SI_3}
    \end{subfigure}
    
    \caption{Visualization of some graph matching results by LinSATNet and Projection layer on Pascal VOC Keypoint dataset. The green lines represent correct matching. The red lines represent incorrect matching. The images in the same column are chosen from the same category.}
    \label{fig: graph_visual_SI_all}
\end{figure}

\newpage
\section{Manifold-constrained hyper-connection}
\label{sec: appendix_mhc}

We details the experimental settings of training mHC-based ViT.

\noindent \textbf{Model.} As for the backbone, we adopt the ViT-base architecture, strictly adhering to the configurations specified in~\cite{dosovitskiy2021an}. The residual connection is replaced by the following hyper-connection \cite{zhu2024hyper}
\begin{equation*}
\label{eq:hyperconn}
y^\top = \mathbf{1}^{\top}_{c} \left( H^{\mathrm{res}} X + H^{\mathrm{post}} \mathcal{F} \left( \left( H^{\mathrm{pre}} X\right)^{\top}, \theta \right)^\top \right),
\end{equation*}
where $X=\mathbf{1}_cx^\top \in \mathbb{R}^{c\times d} $ represents the input information across $c$ paths, $H^{\mathrm{res}} \in \mathbb{R}^{c \times c}$, $H^{\mathrm{pre}} \in \mathbb{R}^{1 \times c}$, $H^{\mathrm{post}} \in \mathbb{R}^{c \times 1}$ are learnable parameters, $y\in \mathbb{R}^c$ is the output, and $\mathcal{F}$ represents the mapping corresponding to the Transformer layer in \cite{dosovitskiy2021an}. To stabilize the training, manifold-constrained hyper-connection (mHC) is proposed to ensure $H^{\mathrm{res}} \in \mathcal{P}_B$, where $\mathcal{P}_B$ is the Birkhoff polytope
\begin{equation}
\label{eq: cons_mhc}
    \mathcal{P}_B = \{H \in \mathbb{R}^{c \times c} ~|~ H\textbf{1}_c = \textbf{1}_c, H^\top \textbf{1} = \textbf{1}_c, H \ge 0\}.
\end{equation}
A finite-step Sinkhorn algorithm \cite{sinkhorn1967concerning} is applied to enforce the linear constraints on $H^{\mathrm{res}}$ in mHC. We conduct the comparative experiments on ViT-base architecture with the residual connection replaced by mHC and the proposed projection layer, respectively. In the experiment, we set $c=8$.

\noindent \textbf{Forward pass of projection layer.} We apply a well-established semismooth Newton algorithm \cite{li2020efficient} to perform the forward pass in this experiment. 

\noindent \textbf{Dataset.} The ILSVRC-2012 ImageNet dataset \cite{deng2009imagenet} with 1k classes and 1.3M images is used for image classification in this experiment.

\noindent \textbf{Evaluation.} The model performance is evaluated by the top-1 accuracy, which is computed by the percentage of images of correct class prediction.

\noindent \textbf{Feasibility computation.} We first reformulate the linear constraints in \eqref{eq: cons_mhc} into a vector form 
\begin{equation*}
    A_1 \mathrm{vec}(H) = \textbf{1}_c, \quad A_2 \mathrm{vec}(H) = \textbf{1}_c, \quad \mathrm{vec}(H) \ge 0,
\end{equation*}
where $\mathrm{vec}(\cdot)$ is a row-major vectorization operator,
\begin{equation*}
    A_1 = I_c \otimes \textbf{1}_c^\top, \quad A_2 = \textbf{1}_c^\top \otimes I_c.
\end{equation*}
For each sample, the feasibility violation is computed by
\begin{equation*}
    V^{\mathrm{all}}=\max\{V_1^{\mathrm{eq}}, V_2^{\mathrm{eq}}, V^{\mathrm{ineq}}\},
\end{equation*}
where
\begin{equation*}
    V_1^{\mathrm{eq}} = \| A_1\mathrm{vec}(H)-\textbf{1}_c\|_2^2, \quad V_2^{\mathrm{eq}} = \|A_2 \mathrm{vec}(H)-\textbf{1}_c\|_2^2, \quad V^{\mathrm{ineq}} = \| \min\{\mathrm{vec}(H),0\} \|_2^2.
\end{equation*}
In the experiment, the feasibility violation is the average of all samples during 1000 steps across all layers.

\noindent \textbf{Other hyperparameters.} The comprehensive hyperparameters for training are demonstrated in Table~\ref{tab: train_param}.

\begin{table}[ht]
    \centering
    \caption{Training hyperparameters for ViT.}
    \vspace{1em} 
    
    \definecolor{tablegray}{gray}{0.85}
    
    \renewcommand{\arraystretch}{1.2}
    
    \begin{tabular}{>{\columncolor{tablegray}\bfseries}l l}
        \toprule
        \rowcolor{tablegray} \text{Hyperparameter} & \textbf{Value} \\
        \midrule
        Learning rate  & 9e-5\\
        Batch size         & 256 (128 per GPU $\times$ 2 GPUs) \\
        Scheduler          & Cosine annealing with linear warmup (300 steps)  \\
        Data augmentation  & Mixup ($\alpha = 0.2$) \\
        Epochs             & 10 \\
        Optimizer          & AdamW ($\beta_1 = 0.9, \beta_2 = 0.999, \epsilon = 1e-8$) \\
        Gradient clipping  & 1.0 \\
        Weight decay       & 0.3 \\
        Dropout            & 0.1 \\
        Precision          & bf16 \\
        \bottomrule
    \end{tabular}
    \label{tab: train_param}
\end{table}

Mixup~\cite{zhang2017mixup}, as specified in Table~\ref{tab: train_param}, is a widely-used domain-agnostic data augmentation technique. It generates virtual training examples by forming convex combinations of original inputs and their corresponding labels. Unlike traditional geometric augmentations (e.g., cropping or flipping), Mixup regularizes the model to maintain linear behavior between training samples. Formally, given two random samples $(x_i, y_i)$ and $(x_j, y_j)$ from the training set, a synthetic sample $(x, y)$ is generated as follows:$$\lambda \sim \text{Beta}(\alpha, \alpha),$$
$$x = \lambda x_i + (1 - \lambda) x_j,$$
$$y = \lambda y_i + (1 - \lambda) y_j.$$

\noindent \textbf{More numerical results.}

We provide additional experimental results to further validate our approach. Beyond the memory efficiency noted in the main manuscript, our proposed projection layer outperforms the Sinkhorn algorithm—even with 30 iterations—in both training loss convergence (Fig.~\ref{fig:train_loss}) and feasibility error (Fig.~\ref{fig:feas_err}). Notably, increasing the Sinkhorn iterations from 20 to 30 incurs substantial additional memory overhead. However, the resulting improvement in feasibility remains marginal, and the impact on training loss convergence is nearly negligible, with both curves appearing virtually indistinguishable.

\begin{figure}[h]
\centering
\begin{subfigure}[b]{0.49\textwidth}
    \centering
    \includegraphics[width=\linewidth]{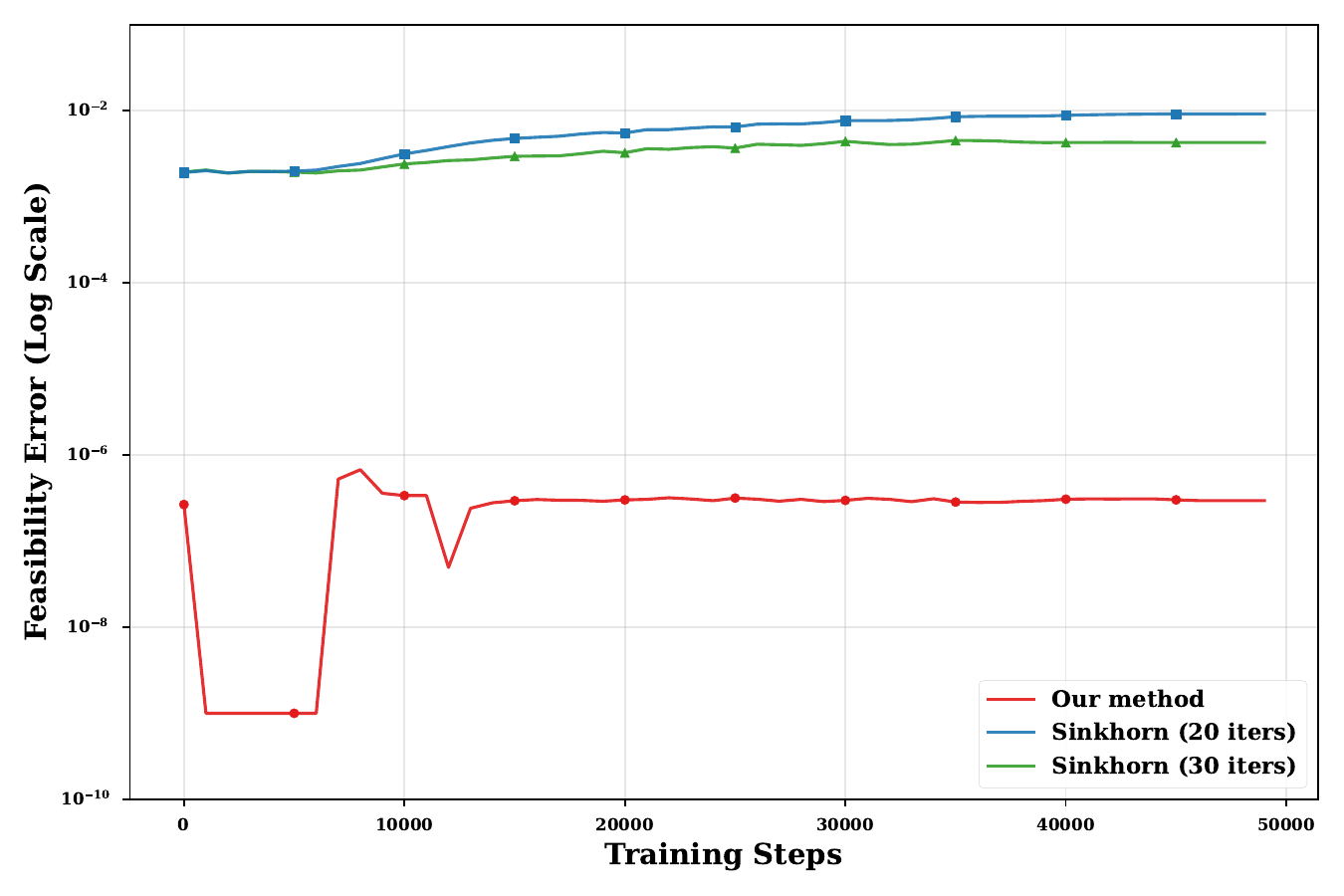}
    \caption{Feasibility error comparison between different algorithms.}
    \label{fig:feas_err}
\end{subfigure}
\hfill
\begin{subfigure}[b]{0.49\textwidth}
    \centering
    \includegraphics[width=\linewidth]{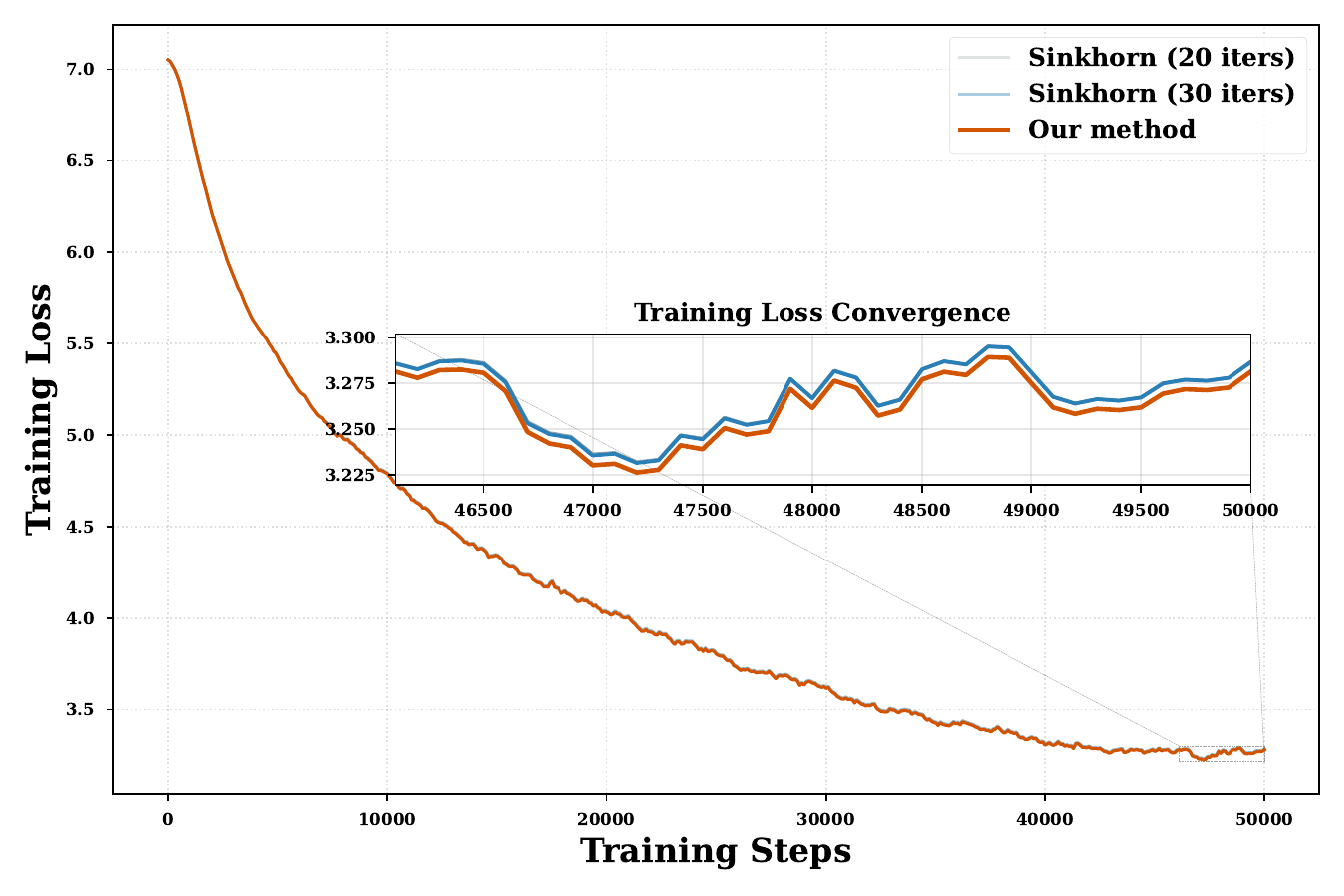}
    \caption{Training loss comparison between different algorithms.}
    \label{fig:train_loss}
\end{subfigure}
\caption{Feasibility error and training loss comparison on ImageNet.}
\label{fig: vit_main_supp}
\end{figure}


Table~\ref{tab:accuracy_comparison} presents the Top-1 accuracy at different training stages. While all three methods achieve comparable accuracy, our proposed projection layer consistently maintains a slight edge over the iterative Sinkhorn baselines. Notably, even when increasing the Sinkhorn iterations to 30, the accuracy does not show a meaningful improvement and occasionally lags behind, further demonstrating the robustness and effectiveness of our projection approach.


\begin{table}[htbp]
\centering
\caption{Comparison of Top-1 Accuracy (\%) throughout the training process.}
\label{tab:accuracy_comparison}
\begin{tabular}{lccc}
\toprule
\textbf{Epoch} & \textbf{Sinkhorn (20 iters)} & \textbf{Sinkhorn (30 iters)} & \textbf{Projection Layer} \\
\midrule
Epoch 1  & 13.98 & 13.95 & \textbf{14.05} \\
Epoch 2  & 22.79 & 22.68 & 22.76 \\
Epoch 3  & 30.15 & 30.18 & \textbf{30.37} \\
Epoch 4  & 35.16 & 35.19 & \textbf{35.37} \\
Epoch 5  & 39.81 & 39.78 & \textbf{39.90} \\
Epoch 6  & 42.84 & 42.74 & \textbf{42.90} \\
Epoch 7  & 45.52 & 45.55 & \textbf{45.74} \\
Epoch 8  & 47.45 & 47.45 & \textbf{47.69} \\
Epoch 9  & 48.43 & 48.41 & \textbf{48.56} \\
Epoch 10 & 48.67 & 48.56 & \textbf{48.73} \\
\midrule
\bottomrule
\end{tabular}
\end{table}

\end{document}